\documentclass[11pt]{article}
\usepackage[T1]{fontenc}
\usepackage{amsthm,amssymb,amsmath,amscd}
\usepackage{wasysym}

\usepackage[all]{xy}
\usepackage{mathrsfs}
\usepackage{multirow}
\usepackage{color}
\usepackage{txfonts}
\usepackage{url}
\usepackage{paralist}

\DeclareMathAlphabet{\mathcal}{OMS}{cmsy}{m}{n}
\DeclareSymbolFont{largesymbols}{OMX}{cmex}{m}{n}

\renewenvironment{itemize}[1]{\begin{compactitem}#1}{\end{compactitem}}

\objectmargin{1pt}

\newtheorem{Def}{Definition}[section]
\newtheorem{Prop}[Def]{Proposition}
\newtheorem{Theo}[Def]{Theorem}
\newtheorem{Lem}[Def]{Lemma}
\newtheorem{Koro}[Def]{Corollary}

\newcommand{\add}{\operatorname{add}}

\newcommand{\Img}{\operatorname{Im}}

\newcommand{\Ker}{\operatorname{Ker}}

\newcommand{\opp}{^{\rm op}}

\newcommand{\Hom}{\operatorname{Hom}}

\newcommand{\HomP}{{\rm Hom}^{\bullet}}

\newcommand{\End}{\operatorname{End}}

\newcommand{\E}{\textsf{E}}
\newcommand{\Oc}[3]{{#1}^{#2,#3}}

\newcommand{\cpx}[1]{#1^{\bullet}}
\newcommand{\D}[1]{\mathscr{D}(#1)}

\newcommand{\Db}[1]{ \mathscr{D}^{\rm b}(#1)}
\newcommand{\C}[1]{\mathscr{C}(#1)}

\newcommand{\Cb}[1]{{\mathscr{C}^b}(#1)}
\newcommand{\K}[1]{\mathscr{K}(#1)}

\newcommand{\Kb}[1]{ \mathscr{K}^{\rm b}(#1)}
\newcommand{\modcat}[1]{#1\mbox{{\rm -mod}}}
\newcommand{\Modcat}[1]{#1\mbox{{\rm -Mod}}}

\newcommand{\pmodcat}[1]{#1\mbox{{\rm -proj}}}

\newcommand{\lra}{\longrightarrow}
\newcommand{\ra}{\rightarrow}
\newcommand{\lraf}[1]{\stackrel{#1}{\lra}}
\newcommand{\raf}[1]{\stackrel{#1}{\ra}}

\def\gh{{\sf gh}}
\def\cogh{{\sf cogh}}
\def\fgh{{\sf F}{\gh}}
\def\fcogh{{{\sf F}\cogh}}

\title{Approximations, ghosts and derived equivalences}
\author{Yiping Chen and Wei Hu$^*$}
\date{}

\begin{document}
\maketitle

\begin{abstract}

Approximation sequences and derived equivalences occur frequently in the research of mutation of tilting objects in representation theory, algebraic geometry and noncommutative geometry.
In this paper, we introduce symmetric approximation sequences in additive categories and weakly $n$-angulated categories which include (higher) Auslander-Reiten sequences (triangles) and mutation sequences in algebra and geometry,
and show that such sequences always give rise to derived equivalences between the quotient rings of endomorphism rings of objects in the sequences modulo some ghost and coghost ideals.

 \end{abstract}

\renewcommand{\thefootnote}{\alph{footnote}}
\setcounter{footnote}{-1} \footnote{ $^*$ Corresponding author.
Email: huwei@bnu.edu.cn; Fax: 0086 10 58808202.}
\renewcommand{\thefootnote}{\alph{footnote}}
\setcounter{footnote}{-1} 
\renewcommand{\thefootnote}{\alph{footnote}}
\setcounter{footnote}{-1} 

\section{Introduction}

Derived categories and derived equivalences   were first introduced by Grothendieck and Verdier in early 1960s.
Derived equivalences between rings preserve many significant invariants such as Hochschild (co)homology, cyclic homology,  center and $K$-theory, etc. In general, it is very hard to tell whether two given rings are derived equivalent or not, and is also very difficult to describe the derived equivalence class of a given ring. One idea is to study derived equivalences ``locally'', that is, to establish some elementary derived equivalences between certain nicely related rings, and hope that most derived equivalent rings can reach each other by a sequence of such elementary derived equivalences. Mutation of objects in categories provide many rings of such kind, where approximations play a central role. The mutation procedure reads as follows: let $T:=M\oplus Y$ be an object in an abelian or a triangulated category, and let
 $$X\raf{g} M'\raf{f}Y$$
be an exact sequence or a triangle such that  $f$ is a right $\add(M)$-approximation. The object $T':=M\oplus X$ is called the right mutation of $T$ at $M$. Dually, one has left mutations. In many cases, the morphism $g$ is also a left $\add(M)$-approximation, and the objects $T$ and $T'$ are left and right mutations of each other at the direct summand $M$. The endomorphism rings $\End(T)$ and $\End(T')$ are then related by this mutation procedure. This occurs in many aspects: mutations of exceptional sequences in the coherent sheaf category of varieties \cite{Rudakov1990}, mutations of tilting modules \cite{Happel1998b}, mutations of cluster tilting objects \cite{Buan2006}, mutations of silting objects \cite{Aihara2012},  and mutations of modifying modules \cite{Iyama2014a} in the study of NCCR conjecture. Auslander-Reiten sequences over artin algebras can also be viewed as mutation sequences.

One can ask whether the rings $\End(T)$ and $\End(T')$ are always derived equivalent or not. For an Auslander-Reiten sequence $0\ra X\ra M\ra Y\ra 0$ over an artin algebra, the endomorphism algebras   $\End(M\oplus X)$ and $\End(M\oplus Y)$ are derived equivalent \cite{Hu2011}. Mutation of tilting modules and modifying modules also provide examples where $\End(T)$ and $\End(T')$ are derived equivalent. However, this is not always true. For instance, the endomorphism algebras of cluster tilting objects related by mutation are not always derived equivalent.  Also,  in general, Auslander-Reiten triangles do not give rise to derived equivalences. In \cite{Hu2013a}, for certain triangles, it was proved that the quotient algebras of $\End(T)$ and $\End(T')$ modulo some particularly defined ideals are still derived equivalent. It remains unclear why these ideals naturally occur.

The aim of this paper is to find the general statement behind this phenomenon. In the mutation procedure, the symmetric (left and right) approximation property plays a central role. We shall introduce the notion of {\em symmetric approximation sequence} (Definition \ref{def-sym-approx-sequence}), which can be viewed as a higher mutation sequence and covers all the known mutation sequences and triangles.

Another main ingredient of our results are ghost and coghost ideals.  Let ${\cal C}$ be an additive category, and let ${\cal D}\subset {\cal C}$ be a full subcategory of ${\cal C}$. The
{\em ${\cal D}$-ghost ideal}, denoted by $\gh_{\cal D}$, is the ideal of ${\cal C}$  consisting of all morphisms $f$ in ${\cal C}$ with ${\cal C}({\cal D}, f)=0$. Dually,  the $\mathcal{D}$-coghost ideal, denoted by $\cogh_{\mathcal{D}}$ is the ideal of $\mathcal{C}$ consisting of morphisms $f$ in $\mathcal{C}$ such that ${\mathcal{C}}(f, \mathcal{D})=0$.

Our main result can be described as the following theorem.

\begin{Theo}[=Theorem \ref{theorem-ghosts}]
Let ${\cal C}$ be an additive category and let $M\in {\cal C}$ be an object. Suppose that
$$X\lraf{f_0} M_1\lraf{f_1}\cdots\lraf{f_{n-1}} M_n\lraf{f_{n}} Y$$
is a symmetric $\add(M)$-approximation sequence (see Definition \ref{def-sym-approx-sequence} below). Then the quotient rings
$$\displaystyle{\frac{\End_{\cal C}(M\oplus X)}{\cogh_M(M\oplus X)}}
\quad \mbox{ and }\quad \displaystyle{\frac{\End_{\mathcal{C}}(M\oplus Y)}{\gh_M(M\oplus Y)}}$$
 are derived equivalent.
\label{theorem-ghost-intro}
\end{Theo}

Let us briefly explain the generality of Theorem \ref{theorem-ghosts}. The notion of symmetric approximation sequences covers several notions in the literature:
\begin{itemize}
\item  Auslander-Reiten sequences, mutation sequences of tilting modules and modifying modules.
\item Mutation triangles of exceptional sequences, cluster tilting objects and silting objects.
\item Higher Auslander-Reiten sequences and Auslander-Reiten $n$-angles (when the starting term does not occur as a direct summand of middle terms).
\end{itemize}
As a consequence, Theorem \ref{theorem-ghost-intro} uniformly generalizes several known results: \cite[Theorem 1.1 and Proposition 3.5]{Hu2011}, \cite[Theorem 5.4]{Geis2006} and \cite[Theorem 5.3]{Ladkani2010}. Let us remark that the ghost ideal and the coghost ideal occurring in Theorem \ref{theorem-ghost-intro} are zero when $f_0$ is monic and $f_{n}$ is epic. In this case, one get a derived equivalence between the endomorphism rings. One can also apply Theorem \ref{theorem-ghost-intro} to get derived equivalences from  Auslander-Reiten $n$-angles.

\smallskip
Next we mention that the ideals occurring in Theorem \ref{theorem-ghost-intro} can be chosen to be some smaller ideals when the category $\mathcal{C}$ has a weak $n$-angulated structure.

As a generalization of triangulated categories, Geiss et al \cite{Geiss2013} introduced $n$-angulated categories, which occur widely in cluster tilting theory and are closely related to algebraic geometry and string theory. An $n$-angulated category is an additive category $\mathcal{C}$ together with an automorphism $\Sigma$ and a class of $n$-angles satisfying four axioms (F1), (F2), (F3) and (F4) (see \cite{Geiss2013} for details). The weak $n$-angulated structure we need in this paper can be obtained by dropping the axiom (F4) (the pushout axiom) and the condition (F1)(c) which says that every morphism in $\mathcal{C}$ can be extended to an $n$-angle.  An additive category $\mathcal{C}$ with this weak $n$-angulated structure will be called a {\em weakly $n$-angulated category} in this paper.   Roughly speaking, the relationship between a weakly $n$-angulated category and an $n$-angulated category is like that between an additive category and an abelian category.  In a weakly $n$-angulated category, we do not have the pushout axiom (Octahedral axiom when $n=3$) and we do not require every morphism to be embedded into an $n$-angle. In an additive category, in general,   pushout/pullback does not  exist, and a morphism does not necessarily have a kernel or cokernel.

Let $\mathcal{C}$ be an additive category, and let $\mathcal{D}$ be an additive subcategory of $\mathcal{C}$. We denote by $\mathcal{F}_{\mathcal{D}}$ the ideal of $\mathcal{C}$ consisting of morphisms factorizing through objects in $\mathcal{D}$. The intersection $\gh_{\mathcal{D}}\cap \mathcal{F}_{\mathcal{D}}$ is called the ideal of  {\em factorizable $\mathcal{D}$-ghosts} of $\mathcal{C}$, denoted by $\fgh_{\mathcal{D}}$. Similarly, the intersection $\cogh_{\mathcal{D}}\cap{\mathcal{F}}_{\mathcal{D}}$ is called the ideal of {\em factorizable $\mathcal{D}$-coghosts} of $\mathcal{C}$, denoted by $\fcogh_{\mathcal{D}}$. It turns out that the ideals defined in \cite{Hu2013a} are actually factorizable ghost and coghost ideals.

\begin{Theo}[=Theorem \ref{theorem-Fghosts}]
  Let $({\cal C}, \Sigma)$ be a weakly $n$-angulated category, and let $M$ be an object in ${\cal C}$. Suppose that
  $$X\lraf{f}M_1\lra\cdots\lra M_{n-2}\lraf{g} Y\lraf{\eta} \Sigma X$$
 is an $n$-angle in ${\cal C}$ such that $f$ and $g$ are left and right $\add(M)$-approximations, respectively. Then the quotient rings
 $$\displaystyle{\frac{\End_{\cal C}(M\oplus X)}{\fcogh_M(M\oplus X)}}
\quad \mbox{ and }\quad \displaystyle{\frac{\End_{\mathcal{C}}(M\oplus Y)}{\fgh_M(M\oplus Y)}}$$
 are derived equivalent.
 \label{theorem-Fghost-intro}
\end{Theo}

Let us remark that,  in Theorem \ref{theorem-Fghost-intro}, the sequence $X\raf{f} M_1\ra\cdots\ra M_{n-2}\raf{g}Y$ is also a symmetric $\add(M)$-approximation sequence, and thus Theorem \ref{theorem-ghost-intro} also  applies.

This paper is organized as follows. In Section 2, we make some preparations, including the $\Phi$-orbit construction, ghosts and coghosts.  Section 3 and 4 are devoted to proving Theorem \ref{theorem-ghost-intro} and Theorem \ref{theorem-Fghost-intro} respectively.  Some examples will be given in the final section.

\section{Preliminary}
In this section, we shall recall basic definitions and facts which
are needed in our proofs.

\subsection{Conventions}
Let $\cal C$ be an additive category. For two objects $X, Y$ in $\mathcal{C}$, we denote by $\mathcal{C}(X, Y)$ the set of morphisms from $X$ to $Y$.  The endomorphism ring ${\cal C}(X, X)$ of an object $X$ is denoted by $\End_{\cal C}(X)$.  We write
$\add_{\cal C}(X)$ for the full subcategory of $\cal C$ consisting of all direct
summands of finite direct sums of copies of $X$.  If there is no confusion, we just write $\add(X)$ for $\add_{\cal C}(X)$.
An object $X$ in
$\cal C$ is called an \emph{additive generator} for $\cal C$ if
${\cal C}$ = add$(X)$.
 For two morphisms $f:X\rightarrow Y$ and
$g:Y\rightarrow Z$ in $\cal C$, we write $fg$ for their composite.  But for two functors $F:\mathcal{C}\rightarrow \mathcal{D}$ and
$G:\mathcal{D}\rightarrow\mathcal{E}$ of categories, we write $GF$ for their composite instead of $FG$.  In this setting, it would be convenient to write $(x)\phi$ for the image of an element $x$ under a map $\phi$ between sets. For a morphism $f: X\lra Y$ and an object  $Z$ in ${\cal C}$, the natural morphism ${\cal C}(Z, f): {\cal C}(Z, X)\lra {\cal C}(Z, Y)$, sending $g$ to $gf$, is denoted by $f_*$. That is, $(g)f_*=gf$ for all $g\in{\cal C}(Z, X)$. Dually, the natural morphism ${\cal C}(f, Z)$ is denoted by $f^*$.

\smallskip
All categories in this paper are additive categories, and all functors are additive functors.  Let $\eta: F\ra G$ be a natural transformation between two functors from ${\cal C}$ to ${\cal D}$.  For an object $X\in {\cal C}$,  we denote by $\eta_X$ the morphism from $F(X)$ to $G(X)$  induced by $\eta$.  For functors  $H: {\cal A}\ra {\cal C}$ and $L: {\cal D}\lra {\cal E}$, we have a natural transformation $L\eta H: LFH\lra LGH$ induced by $\eta$.

\smallskip
Let ${\cal C}$ be a category. A functor $F$ from ${\cal C}$ to itself is called an {\em endo-functor} of ${\cal C}$.  If there is another endo-functor $G$ of ${\cal C}$ such that $FG=GF=id_{\cal C}$, where $id_{\cal C}$ is the identity functor of ${\cal C}$, then $F$ is called an {\em automorphism} of ${\cal C}$.  A endo-functor $F$ is called an {\em auto-equivalence} provided that there is another endo-functor $G$ of ${\cal C}$ such that both $FG$ and $GF$ are naturally isomorphic to $id_{\cal C}$.

\subsection{Complexes and derived equivalences}
Let ${\cal C}$ be an additive category. A complex $\cpx{X}$ over ${\cal C}$ is a sequence of morphisms $\cdots\ra X^{i-1}\raf{d_X^{i-1}}X^i\raf{d_X^i} X^{i+1}\raf{d_X^{i+1}}\cdots$ between objects in ${\cal C}$ such that $d_X^id_X^{i+1}=0$ for all $i\in\mathbb{Z}$.  The category of complexes over ${\cal C}$ with morphisms being chain maps is denoted by $\C{\cal C}$.  The homotopy category of complexes over ${\cal C}$ is denoted by $\K{\cal C}$. If ${\cal C}$ is an abelian category, then the derived category of complexes over ${\cal C}$ is denoted by $\D{\cal C}$.  We write $\Cb{\cal C}$, $\Kb{\cal C}$ and $\Db{C}$ respectively for the full subcategories of $\C{\cal C}$, $\K{\cal C}$ and $\D{\cal C}$ consisting of bounded complexes.

It is well known that the categories $\K{\cal C}$ and $\D{\cal C}$ are triangulated categories with $\Kb{\cal C}$ and $\Db{\cal C}$ being their full triangulated  subcategories, respectively.  For basic results on triangulated categories, we refer to Happel's book \cite{Happel1988}. However,  the shifting functor in a triangulated category is written as $\Sigma$ in this paper.

For two complexes $\cpx{X}$ and $\cpx{Y}$ over ${\cal C}$, we write $\HomP_{\cal C}(\cpx{X}, \cpx{Y})$ for the total complex of the double complex with the $(i, j)$-term ${\cal C}(X^{-j}, Y^i)$.

\medskip
Let $\Lambda$ be a ring with identity.  The category of left $\Lambda$-modules, denoted by $\Modcat{\Lambda}$, is an abelian category.  The full subcategory of $\Modcat{\Lambda}$ consisting of finitely generated projective $\Lambda$-modules is denoted by $\pmodcat{\Lambda}$.
Following \cite{Rickard1989a}, two rings $\Lambda$ and $\Gamma$ are said to be {\em derived equivalent} provided that the derived categories $\Db{\Modcat{\Lambda}}$  and $\Db{\Modcat{\Gamma}}$ of bounded complexes are equivalent as triangulated categories. Due to the work of Rickard \cite{Rickard1989a} (see also \cite{Keller1994}), two rings $\Lambda$ and $\Gamma$ are derived equivalent if and only if there is a bounded complex $\cpx{T}$ of finitely generated projective $\Lambda$-modules satisfying the following two conditions,

\smallskip
(a). $\cpx{T}$ is self-orthogonal, that is, $\Kb{\Lambda}(\cpx{T}, \Sigma^i\cpx{T})=0$ for all $i\neq 0$;

(b). $\add(\cpx{T})$ generates $\Kb{\pmodcat{\Lambda}}$ as a triangulated category,

\medskip
{\parindent=0pt such that } $\End_{\Kb{\pmodcat{\Lambda}}}(\cpx{T})$ is isomorphic to $\Gamma$ as rings.  A complex $\cpx{T}$  in $\Kb{\pmodcat{\Lambda}}$ satisfying the above two conditions is called a {\em tilting complex } over $\Lambda$.

\subsection{Admissible sets and $\Phi$-orbit categories}

Let us recall from \cite{Hu2013} and \cite{Hu2013a} the definition of admissible subsets.  A subset $\Phi$ of $\mathbb{Z}$ containing $0$ is called an {\em admissible subset}  provided that the following property holds:

\medskip
   {\em If $i+j+k\in\Phi$ for three elements $i,j,k$ in $\Phi$, then $i+j\in\Phi$ if and only if $j+k\in\Phi$}

\medskip
{\parindent=0pt Typical} examples of admissible subsets of $\mathbb{Z}$ include $n\mathbb{Z}$ and $\{0, 1, \cdots, n\}$.  Suppose that $\Phi$ is an admissible subset of $\mathbb{Z}$. Then $-\Phi:=\{-i|i\in\Phi\}$, $\Phi^{\geq 0}:=\{i\in\Phi| i\geq 0\}$ and $\Phi^{\leq 0}:=\{i\in\Phi|i\leq 0\}$ are all admissible. Let $m$ be an integer. The set $m\Phi:=\{mi|i\in\Phi\}$ is admissible. Moreover, if $m\geq 3$, then the set $\Phi^m:=\{i^m|i\in\Phi\}$ is admissible. Nevertheless, not all subsets of $\mathbb{Z}$ containing zero are admissible. For instance, the set $\{0, 1, 2, 4\}$ is not admissible.

\medskip
Now let ${\cal T}$ be an additive category, and let $F$ be an endo-functor of  ${\cal T}$. If $F$ is not an equivalence, we set $F^i=0$ for all $i<0$. If $F$ is an equivalence, we fix a quasi-inverse $F^{-1}$ of $F$, and set $F^i:=(F^{-1})^{-i}$ for $i<0$.  The functor $F^0$ is defined to be the identity functor on ${\cal T}$.  We can define a category $\Oc{\cal T}{F}{\Phi}$  as follows. The objects in $\Oc{\cal T}{F}{\Phi}$ are the same as ${\cal T}$, and the morphism space $\Oc{\cal T}{F}{\Phi}(X, Y)$ for two objects $X, Y$ is defined to be
$$\bigoplus_{i\in\Phi}{\cal T}(X, F^iY).$$
In \cite{Hu2013a}, for each pair of integers $u$ and $v$, a natural transformation $\chi(u, v)$ from $F^uF^v$ to $F^{u+v}$ is defined, and it is proved that the composition
$${\cal T}(X, F^uY)\times {\cal T}(Y, F^vZ)\lra {\cal T}(X, F^{u+v}Z), $$
sending $(f_u, g_v)$ to $f_u*g_v:=f_uF^{u}(g_v)\chi(u,v)_Z$, is associative.  We refer to \cite[Section 2.3]{Hu2013a} for the details of the natural transformations $\chi(u, v): F^uF^v\lra F^{u+v}$.  As a result, for morphisms $f=(f_i)_{i\in\Phi}\in \Oc{\cal T}{F}{\Phi}(X, Y)$ and $g=(g_i)_{i\in\Phi}\in\Oc{\cal T}{F}{\Phi}(Y, Z)$, the composition
$$(f, g)\mapsto fg:=\left(\sum_{{u,v\in\Phi}\atop {u+v=i}}f_u*g_v\right)_{i\in\Phi}$$
is associative.  Thus $\Oc{\cal T}{F}{\Phi}$ is indeed an additive category, and is called the {\em  {$\Phi$-orbit category}} of ${\cal T}$ under the functor $F$.  The endomorphism ring of an object $X$ in $\Oc{\cal T}{F}{\Phi}$ is denoted by $\E_{\cal T}^{F, \Phi}(X)$, and is called the {\em $\Phi$-Yoneda algebra} of $X$ with respect to $F$.

For each $X, Y\in{\cal T}$, the morphism space $\Oc{\cal T}{F}{\Phi}(X, Y)=\bigoplus_{i\in\Phi}{\cal T}(X, F^iY)$ is $\Phi$-graded.  Every morphism $\alpha\in {\cal T}(X, F^iY)$ can be viewed as a homogeneous morphism in $\Oc{\cal T}{F}{\Phi}(X, Y)$ of degree $i$.

Suppose that $F$ is an auto-equivalence of ${\cal T}$. If both $i$ and $-i$ are in the admissible subset $\Phi$, then $X$ and $F^iX$ are isomorphic in the $\Phi$-orbit category $\Oc{\cal T}{F}{\Phi}$.  Actually, let $f:=\chi(-i, i)_X^{-1}: X\lra F^{-i}(F^iX)$ and $g:=1_{F^iX}: F^iX\lra F^iX$. Considering $f$ as a homogeneous morphism in $\Oc{\cal T}{F}{\Phi}$ of degree $-i$, and $g$ as a homogeneous morphism in $\Oc{\cal T}{F}{\Phi}$ of degree $i$, we have $f*g=1_X$ and $g*f=1_{F^iX}$.

\subsection{Approximations and cohomological approximations}

\medskip
Now we recall some definitions from \cite{Auslander1980}.

\medskip
Let $\cal C$ be a category, and let $\cal D$ be a full subcategory
of $\cal C$, and $X$ an object in $\cal C$. A morphism $f: D\lra X$
in $\cal C$ is called a \emph{right} $\cal D$-\emph{approximation}
of $X$ if $D\in {\cal D}$ and the induced map Hom$_{\cal C}(D',f)$:
Hom$_{\cal C}(D',D)\lra$ Hom$_{\cal C}(D',X)$ is surjective for
every object $D'\in {\cal D}$.  Dually, there is the notion of a \emph{left} $\cal
D$-\emph{approximation}. The subcategory $\cal D$ is called
\emph{contravariantly} (respectively, \emph{covariantly})
\emph{finite} in $\cal C$ if every object in $\cal C$ has a right
(respectively, left) $\cal D$-approximation. The subcategory $\cal
D$ is called \emph{functorially finite} in $\cal C$ if $\cal D$ is
both contravariantly and covariantly finite in $\cal C$.

\medskip
Cohomological approximations were introduced in \cite{Hu2013a}.
Let ${\cal T}$ be an additive category, and let $F$ be a functor from ${\cal T}$ to itself.  Suppose that $\Phi$ is a non-empty subset of $\mathbb{Z}$, and that ${\cal D}$ is a full additive subcategory of ${\cal T}$.  A morphism $f: X\ra D^X$ in ${\cal T}$ with $D^X\in {\cal D}$ is called a {\em left $({\cal D}, F, \Phi)$-approximation} if every morphism $X\ra F^iD$, where $D\in{\cal D}$ and $i\in\Phi$, factorizes through $f$.   In case that $\Phi$ is an admissible subset,  we have the $\Phi$-orbit category $\Oc{\cal T}{F}{\Phi}$, and  that $f$ is a left $({\cal D}, F, \Phi)$-approximation is equivalent to saying that $\Oc{\cal T}{F}{\Phi}(D^X, {\cal D})\ra \Oc{\cal T}{F}{\Phi}(X, {\cal D})$ is surjective, i.e., the morphism $f$, as a homogeneous morphism of degree zero, is a left ${\cal D}$-approximation in the orbit category $\Oc{\cal T}{F}{\Phi}$.

\medskip
In \cite{Hu2013a}, a right $({\cal D}, F, \Phi)$-approximation is defined to be a morphism $g: D_Y\ra Y$ in ${\cal T}$ with $D_Y\in {\cal D}$ such that every morphism from $F^iD$ to $Y$ with $i\in\Phi$ and $D\in {\cal D}$ factorizes through $g$. Unfortunately, this does not fit the $\Phi$-orbit category well: $g$ is NOT a right ${\cal D}$-approximation in the orbit category $\Oc{\cal T}{F}{\Phi}$ in general.  However, when $F$ is an auto-equivalence with a quasi-inverse $F^{-1}$, a  right $({\cal D}, F, -\Phi)$-approximation is still a right ${\cal D}$-approximation in $\Oc{\cal T}{F}{\Phi}$.  Here we re-define a right $({\cal D}, F, \Phi)$-approximation as follows.

\medskip
 {\it A morphism $g: D_Y\ra Y$ in ${\cal T}$ with $D_Y\in {\cal D}$ is called a {\em right $({\cal D}, F, \Phi)$-approximation} if every morphism from $D$ to $F^iY$ with $i\in\Phi$ and $D\in {\cal D}$ factorizes through $F^ig$.  }

\medskip
{\parindent=0pt Suppose} that $\Phi$ is an admissible subset of $\mathbb{Z}$. Thus, a morphism $g: D_Y\ra Y$ in ${\cal T}$ is a right $({\cal T}, F, \Phi)$-approximation if and only if $g$ is a right ${\cal D}$-approximation in the $\Phi$-orbit category $\Oc{\cal T}{F}{\Phi}$, no matter $F$ is an equivalence or not.

\subsection{Ghosts and factorizable ghosts}
\label{subsection-ghosts}

Let   ${\cal C}$ be an additive category.  By an {\em ideal} $\mathcal{I}$ of ${\cal C}$ we mean additive subgroups ${\cal I}(A, B)\subseteq \Hom_{\cal C}(A, B)$ for all $A$ and $B$ in ${\cal C}$, such that the composite $\alpha\beta$ belongs to ${\cal I}$ provided either $\alpha$ or $\beta$ is in ${\cal I}$. We denote ${\cal I}(A, A)$ simply by ${\cal I}(A)$. The quotient category ${\cal C}/{\cal I}$ of ${\cal C}$ modulo an ideal ${\cal I}$ has the same objects as ${\cal C}$ and has morphism space $({\cal C}/{\cal I})(A, B):={\cal C}(A, B)/{\cal I}(A, B)$ for two objects $A$ and $B$.

Let ${\cal D}$ be a full additive subcategory of ${\cal C}$, a morphism $f$ in ${\cal C}$ is called a {\em ${\cal D}$-ghost} provided that ${\cal C}({\cal D}, f)=0$. All ${\cal D}$-ghosts in ${\cal C}$ form an ideal of ${\cal C}$, called the {\em ideal of ${\cal D}$-ghosts} and denoted by $\gh_{\cal D}$.  Dually, a morphism $g$ in ${\cal C}$ is called a {\em ${\cal D}$-coghost} if ${\cal C}(g, {\cal D})=0$, and the ideal consisting of all ${\cal D}$-coghosts is dented by $\cogh_{\cal D}$.

Let $\mathcal{F}_{\mathcal{D}}$ be the ideal of morphisms in $\mathcal{C}$ factorizing through objects in $\mathcal{D}$.  The intersection $\gh_{\mathcal{D}}\cap\mathcal{F}_{\mathcal{D}}$ is called the {\em ideal of factorizable $\mathcal{D}$-ghosts}  of $\mathcal{C}$, denoted by $\fgh_{\mathcal{D}}$.  Similarly, the intersection $\cogh_{\mathcal{D}}\cap\mathcal{F}_{\mathcal{D}}$  is called the {\em ideal of factorizable $\mathcal{D}$-coghosts} of $\mathcal{C}$, denoted by $\fcogh_{\mathcal{D}}$.

\begin{Lem}
  Keeping the notations above, we have the following:

  $(1)$. If $A\in{\cal C}$ admits a right ${\cal D}$-approximation $f_A: D_A\ra A$, then
      $$\gh_{\cal D}(A, B)=\{g\in {\cal C}(A, B)\, |\, f_Ag=0\}.$$

 $(2)$ If $B\in{\cal C}$ admits a left ${\cal D}$-approximation $f^B: B\ra D^B$, then
  $$\cogh_{\cal D}(A, B)=\{g\in {\cal C}(A, B)\, |\, gf^B=0\}.$$

  $(3)$.  If $A\in {\cal D}$, then $\gh_{\cal D}(A, B)=0$ and $\cogh_{\cal D}(A, B)=\fcogh_{\cal D}(A, B)$.

  $(4)$. If $B\in {\cal D}$, then $\cogh_{\cal D}(A, B)=0$ and $\gh_{\cal D}(A, B)=\fgh_{\cal D}(A, B)$.
\label{lemma-ann}
\end{Lem}

 \begin{proof}  (1). Let $g: A\ra B$ be  in $\gh_{\cal D}(A, B)$. Then ${\cal C}({\cal D}, g)=0$, and particularly ${\cal C}(D_A, g)=0$. Consequently $f_Ag=0$. Conversely, let $g$ be in ${\cal C}(A, B)$ such that $f_Ag=0$. It follows that
 $$0={\cal C}(D', f_Ag)={\cal C}(D', f_A)\cdot{\cal C}(D', g)$$ for all $D'\in {\cal D}$. Moreover, since $f_A$ is a right ${\cal D}$-approximation, the morphism ${\cal C}(D', f_A)$ is surjective. Hence ${\cal C}(D', g)=0$ for all $D'\in {\cal D}$, that is, $g\in\gh_{\cal D}(A, B)$.

\smallskip
The proof of (2) is dual to that of (1).

\smallskip
(3). Suppose that $A\in {\cal D}$. The identity map $1_A: A\ra A$ is a right ${\cal D}$-approximation. It follows from (1) that $\gh_{\cal D}(A, B)=0$.
Clearly, all the morphisms in ${\cal C}(A, B)$ factorizes through the object $A$, which is in ${\cal D}$. This implies that $\fcogh_{\cal D}(A, B)=\cogh_{\cal D}(A, B)$. Similarly, we can prove (4).
\end{proof}

%%%%%%%%%%%%%%%%%%%%%%%%%%%%%%%%%%%%%%
\section{Symmetric approximation sequences and derived equivalences}

In this section, we introduce symmetric approximation sequences in additive categories and show that such a sequence always gives rise to a derived equivalence between the quotient rings of certain endomorphism rings modulo ghosts or coghosts (Theorem \ref{theorem-ghosts}). Among the examples of symmetric approximation sequences are  ${\cal D}$-split sequences, ${\cal D}$-split triangles, mutation triangles in cluster tilting theory, higher Auslander-Reiten sequences and higher Auslander-Reiten triangles

\medskip
Let $\mathcal{C}$ be an additive category, and let ${\cal D}$ be an additive full subcategory of ${\cal C}$.  A  {\em right ${\cal D}$-approximation sequence} in ${\cal C}$ is a sequence
  $$D_m\lra D_{m-1}\lra\cdots\lra D_0\ra Y$$
with $D_i\in {\cal D}$ for all $i=0, \cdots, m$ such that applying ${\cal C}(D, -)$ to the sequence results in an exact sequence
 $${\cal C}(D, D_m)\lra {\cal C}(D, D_{m-1})\lra\cdots\lra {\cal C}(D, D_0)\ra {\cal C}(D, Y)\lra 0$$
for all $D\in {\cal D}$. One can define  {\em left ${\cal D}$-approximation sequences} dually.

\begin{Def}
Let ${\cal C}$ be an additive category and let ${\cal D}$ be a  full additive subcategory of ${\cal C}$. A sequence
\begin{align}
X\lraf{f_0}D_1\lraf{f_1}\cdots\lraf{f_{n-1}}D_{n}\lraf{f_{n}}Y\tag{$\dagger$}
\end{align}
 in ${\cal C}$ is called a {\bf symmetric ${\cal D}$-approximation sequence} if the following three conditions are satisfied.

{\rm (1)}  The sequence $D_1\lraf{f_1}\cdots\lraf{f_{n-1}}D_{n}\lraf{f_{n}}Y$ is a right ${\cal D}$-approximation sequence;

{\rm (2)} The sequence $X\lraf{f_0}D_1\lraf{f_1}\cdots\lraf{f_{n-1}}D_{n}$ is a left ${\cal D}$-approximation sequence;

{\rm (3)} The morphism $f_0$ is a pseudo-kernel of $f_1$,  and $f_{n}$ is a pseudo-cokernel of $f_{n-1}$.
\label{def-sym-approx-sequence}
\end{Def}

In Definition \ref{def-sym-approx-sequence}, if  we replace the condition (3) with the following condition

\medskip
(3') {\em The morphism $f_0$ is a kernel of $f_1$, and $f_{n}$ is a cokernel of $f_{n-1}$, }

\medskip
\noindent
then the sequence $(\dagger)$ is called a {\bf higher $\mathcal{D}$-split sequence}. Comparing with the definition of $\mathcal{D}$-split sequence \cite[Definition 3.1]{Hu2011},  a $\mathcal{D}$-split sequence is precisely a sequence $(\dagger)$ with $n=1$ satisfying the conditions (1), (2) and ($3'$) above.

\medskip

The main result in this section  is the following theorem.

\begin{Theo}
Let ${\cal C}$ be an addtive category and let $M\in {\cal C}$ be an object. Suppose that
$$X\lraf{f_0} M_1\lraf{f_1}\cdots\lraf{f_{n-1}} M_n\lraf{f_{n}} Y$$
is a symmetric $\add(M)$-approximation sequence. Then the quotient rings
$$\displaystyle{\frac{\End_{\cal C}(M\oplus X)}{\cogh_M(M\oplus X)}}
\quad \mbox{ and }\quad \displaystyle{\frac{\End_{\mathcal{C}}(M\oplus Y)}{\gh_M(M\oplus Y)}}$$
 are derived equivalent.
\label{theorem-ghosts}
\end{Theo}

The following lemma and its corollary will be useful in the proof of Theorem \ref{theorem-ghosts}.

\begin{Lem}
Let ${\cal C}$ be an additive category, and let  $M$ be an object in ${\cal C}$. Suppose that
$\cpx{P}$:
$$0\lra P^{0}\lraf{d^{0}}P^{1}\lra\cdots\lra P^{n-1}\lraf{d^{n-1}} P^n\lra 0$$
is a complex over ${\cal C}$ such that $P^i\in\add(M)$ for all $i>0$, and  that the following two conditions are satisfied:

\smallskip
$(1)$. $H^i(\HomP_{\cal C}(M, \cpx{P}))=0$ for all $i\neq 0, n$;

$(2)$.  $H^i(\HomP_{\cal C}(\cpx{P}, M))=0$ for all $i\neq -n$.

\smallskip
{\parindent=0pt Then} $\cpx{P}$ is self-orthogonal as a complex both in $\Kb{{\cal C}/\cogh_M}$ and in $\Kb{{\cal C}/\fcogh_M}$.
\label{Lemma-tilting-complex-left}
\end{Lem}

\begin{proof}
For simplicity, we denote by $\overline{\cal C}$ the category ${\cal C}/\cogh_M$, and denote by $\overline{\overline{\cal C}}$ the category ${\cal C}/\fcogh_M$.

If $n=0$, then the problem is trivial. Now we assume that $n>0$.

   It follows from our assumption (2) that $H^0(\HomP_{\cal C}(\cpx{P}, M))=0$, and consequently the map ${\cal C}(d^0, M): {\cal C}(P^1, M)\ra {\cal C}(P^0, M)$ is surjective. Thus, the morphism $d^0$ is a left $\add(M)$-approximation. By Lemma \ref{lemma-ann} (2), one has $\cogh_M(M, P^0)=\{f\in {\cal C}(M, P^0)|fd^0=0\}=\Ker\,{\cal C}(M, d^0)$.   Moreover,  it follows from Lemma \ref{lemma-ann} (3) that $\cogh_M(M, P^0)=\fcogh_M(M, P^0)$.  Hence the canonical functors ${\cal C}\ra \overline{\cal C}\ra\overline{\overline{\cal C}}$ induce isomorphisms
   $${\cal C}(M, P^0)/\Ker\,{\cal C}(M, d^0)\lraf{\pi^0}\overline{C}(M, P^0)\lraf{p^0}\overline{\overline{C}}(M, P^0).$$
Note that for each $i>0$, by Lemma \ref{lemma-ann} (4), we have $\cogh_M(M,P^i)=0=\fcogh_M(M, P^i)$ since $P^i\in\add(M)$. Thus, for each $i>0$, the canonical functors  ${\cal C}\ra \overline{\cal C}\ra\overline{\overline{\cal C}}$ also induce isomorphisms
  $${\cal C}(M, P^i)\lraf{\pi^i}\overline{C}(M, P^i)\lraf{p^i}\overline{\overline{C}}(M, P^i).$$
In this way, we see that the complexes $\HomP_{\overline{C}}(M, \cpx{P})$ and
  $\HomP_{\overline{\overline{C}}}(M, \cpx{P})$ are both isomorphic to the complex
  $$0\lra{\cal C}(M, P^0)/\Ker\,{\cal C}(M, d^0)\lra {\cal C}(M, P^1)\lra \cdots\lra {\cal C}(M, P^n)\lra 0.$$
 By assumption (1), the above complex has zero homology for all degrees not equal to $n$. Hence
 $$H^i(\Hom_{\overline{\cal C}}(M, \cpx{P}))=0=H^i(\HomP_{\overline{\overline{C}}}(M, \cpx{P}))$$ for all $i\neq n$.
 By Lemma \ref{lemma-ann} (4), we have $\cogh_M(P^i, M)=0$ for all $i$, and therefore $\fcogh_M(P^i, M)=0$ for all $i$. Hence the complexes $\HomP_{{\overline{\cal C}}}(\cpx{P}, M)$, $\HomP_{\overline{\overline{\cal C}}}(\cpx{P}, M)$ and $\HomP_{{\cal C}}(\cpx{P}, M)$ are all isomorphic. Hence $H^i(\HomP_{{\overline{\cal C}}}(\cpx{P}, M))=0=H^i(\HomP_{\overline{\overline{\cal C}}}(\cpx{P}, M))$ for all $i\neq -n$ by assumption (2). The lemma then follows from the dual version of the result \cite[Lemma 2.1]{Hoshino2003}.
 \end{proof}

 \begin{Koro}
 Let $\mathcal{C}$ be an additive category and let $M$ be an object in $\mathcal{C}$. Suppose that
 $$X\lraf{f_0}M_1\lra\cdots\lra M_n\lraf{f_{n}} Y$$
 is a symmetric $\add(M)$-approximation sequence in $\mathcal{C}$. Then the complex
 $$0\lra X\lraf{f_0}M_1\lra\cdots\lra M_n\lra 0$$
 is  self-orthogonal as a complex both in $\Kb{{\cal C}/\cogh_M}$ and in $\Kb{{\cal C}/\fcogh_M}$.
 \label{corollary-sym-seq-tilting}
 \end{Koro}
 \begin{proof}
 Taking $\cpx{P}$ to be the complex $0\ra X\ra M_1\ra\cdots\ra M_n\ra 0$ with $X$ in degree zero, by the conditions in Definition \ref{def-sym-approx-sequence}, the complex $\cpx{P}$  satisfies the conditions of Lemma \ref{Lemma-tilting-complex-left}, and the corollary then follows.
 \end{proof}

 \begin{proof}[Proof of  Theorem  \ref{theorem-ghosts}]
 Note that the quotient rings in the theorem are precisely the endomorphism rings of $M\oplus X$ and $M\oplus Y$ in $\mathcal{C}/\cogh_M$ and $\mathcal{C}/\gh_M$ respectively.

 By the definition of symmetric approximation sequences, the sequence
$$0\lra X\lra M_1\lra\cdots\lra M_n\oplus M\lra Y\oplus M\lra 0$$
 is again a symmetric $\add(M)$-approximation sequence.
  Let $\cpx{T}$ be the complex
  $$0\lra X\lra M_1\lra\cdots\lra M_n\oplus M\lra 0$$
with $X$ in degree zero.  Then it follows from Corollary \ref{corollary-sym-seq-tilting} that $\cpx{T}$ is self-orthogonal in $\Kb{{\cal C}/\cogh_M}$.   Using the full embedding
$$\HomP_{{\cal C}/\cogh_M}(M\oplus X, -): \Kb{\add_{{\cal C}/\cogh_M}(M\oplus X)}\lra \Kb{\pmodcat{\End_{{\cal C}/\cogh_M}(M\oplus X)}}, $$
we see that $\cpx{\tilde{T}}:=\HomP_{{\cal C}/\cogh_M}(M\oplus X, \cpx{T})$ is self-orthogonal in $\Kb{\pmodcat{\End_{{\cal C}/\cogh_M}(M\oplus X)}}$. Moreover, $\add(\cpx{\tilde{T}})$ clearly generates $\Kb{\pmodcat{\End_{{\cal C}/\cogh_M}(M\oplus X)}}$ as a triangulated category.  Hence $\cpx{\tilde{T}}$ is a tilting complex over
$\End_{{\cal C}/{\cogh_M}}(M\oplus X)$ with endomorphism ring isomorphic to $\End_{\Kb{{\cal C}/\cogh_M}}(\cpx{T})$.

 By \cite[Theorem 6.4]{Rickard1989a},  it  remains to prove that $\End_{\Kb{{\cal C}/\cogh_M}}(\cpx{T})$ and  $\End_{{\cal C}/\gh_M}(Y\oplus M)$ are isomorphic. Instead of giving a ring isomorphism directly, we prove that there is a surjective ring homomorphism from $\End_{\Cb{{\cal C}}}(\cpx{T})$ to each of the rings,  and show that these two ring homomorphisms have the same kernel.

We denote the differentials of $\cpx{T}$ by $d^i$ and denote by $\tilde{d}^n$ the map $M_n\oplus M\raf{\left[\begin{smallmatrix}f_n &\\ &1\end{smallmatrix}\right]} Y\oplus M$.

Firstly, we show that there is a surjective ring homomorphism $$\theta: \End_{\Cb{\cal C}}(\cpx{T})\lra \End_{{\cal C}/\gh_M}(Y\oplus M).$$
For each chain map $\cpx{g}: \cpx{T}\lra \cpx{T}$ in $\Cb{\cal C}$, since $\tilde{d}^n$ is a pseudo-cokernel of $d^{n-1}$,  there is a morphism $g\in\End_{\cal C}(Y\oplus M)$ such that the following diagram is commutative
$$\xymatrix@M=1mm{
 X\ar[r]^{d^0}\ar[d]_{g^0} & M_1\ar[r]^{d^1}\ar[d]_{g^1} &\cdots\ar[r]^{d^{n-1}} & M_n\oplus M\ar[r]^{\tilde{d}^n}\ar[d]_{g^n} & Y\oplus M\ar@{-->}[d]_{g} \\
 X\ar[r]^{d^0} & M_1\ar[r]^{d^1} &\cdots\ar[r]^{d^{n-1}} & M_n\oplus M\ar[r]^{\tilde{d}^n} & Y\oplus M.
}$$
Moreover, if $g'$ is another morphism in $\End_{\cal C}(Y\oplus M)$ such that $\tilde{d}^ng'=g^n\tilde{d}^n$, then $\tilde{d}^n(g-g')=0$. By definition $\tilde{d}^n$ is a right $\add(M)$-approximation of $Y\oplus M$. Thus, by Lemma \ref{lemma-ann} (1), the morphism $g-g'$ belongs to $\gh_M(Y\oplus M)$. We denote by $\bar{g}$ the corresponding morphism of $g$ in ${\cal C}/\gh_M$.    Defining $\theta(\cpx{g}):=\bar{g}$ gives rise to a ring homomorphism $\theta$ from
$\End_{\Cb{\cal C}}(\cpx{T})$ to $\End_{{\cal C}/\gh_M}(Y\oplus M)$.  We claim that $\theta$ is surjective. Actually, for each $g\in\End_{\cal C}(Y\oplus M)$, it follows from Definition \ref{def-sym-approx-sequence}(1)(3) that there are morphisms $g^i, i=1,\cdots, n$ such that $g^n\tilde{d}^n=\tilde{d}^ng$ and $g^kd^k=d^kg^{k+1}$ for all $k=1,\cdots, n-1$. Since $d^0$ is a pseudo-kernel of $d^1$ by definition, we get a morphism $g^0: X\ra X$ such that $g^0d^0=d^0g^1$. Thus we get a chain map $\cpx{g}$ in $\End_{\Cb{\cal C}}(\cpx{T})$ such that $\theta(\cpx{g})=\bar{g}$.

Secondly, we claim that there is a surjective ring homomorphism
 $$\varphi: \End_{\Cb{\cal C}}(\cpx{T})\lra \End_{\Kb{{\cal C}/\cogh_M}}(\cpx{T}). $$
Actually, we can define $\varphi$ to be the composite of the ring homomorphism from  $$\End_{\Cb{\cal C}}(\cpx{T})\lra \End_{\Cb{{\cal C}/\cogh_M}}(\cpx{T}),$$ induced by the canonical functor ${\cal C}\ra {\cal C}/\cogh_M$, and the canonical surjective ring homomorphism from
$$\End_{\Cb{{\cal C}/\cogh_M}}(\cpx{T})\lra \End_{\Kb{{\cal C}/\cogh_M}}(\cpx{T}).$$
 Let $g^i: T^i\ra T^i, i=0, 1,\cdots, n$ be morphisms in ${\cal C}$  such that $\cpx{g}$ is a chain map in $\Cb{{\cal C}/\cogh_M}$. Then $g^id^i-d^ig^{i+1}: T^i\ra T^{i+1}$ is in $\cogh_M$ for all $i=0, 1, \cdots, n-1$.  Since $T^i\in\add(M)$ for all $i>0$, by Lemma \ref{lemma-ann} (4), we get $g^id^i-d^ig^{i+1}=0$ for all $i=0, 1, \cdots, n-1$. Hence $\cpx{g}$ is a chain map in $\Cb{\cal C}$, and the canonical map from  $\End_{\Cb{\cal C}}(\cpx{T})$ to $ \End_{\Cb{{\cal C}/\cogh_M}}(\cpx{T})$ is surjective. Consequently $\varphi$ is a surjective ring homomorphism.

Finally, we show that $\theta$ and $\varphi$ have the same kernel, which would result in an isomorphism between $ \End_{\Kb{{\cal C}/\cogh_M}}(\cpx{T})$ and $\End_{{\cal C}/\gh_M}(Y\oplus M)$. By definition, a chain map $\cpx{g}: \cpx{T}\lra \cpx{T}$ is in $\Ker\varphi$ if and only if there exist $h^i: T^i\ra T^{i-1}, i=1,\cdots, n$ in ${\cal C}$ such that $g^n-h^nd^{n-1}$, $g^0-d^0h^1$ and $g^i-h^id^{i-1}-d^ih^{i+1}$ are all in $\cogh_M$ for $i=1,\cdots, n-1$.  Using the fact that $T^i\in\add(M)$ for all $i>0$ and that $d^0$ is a left $\add(M)$-approximation of $X$,  one can show, by Lemma \ref{lemma-ann}, that this is equivalent to saying that $g^n-h^nd^{n-1}=0$, $(g^0-d^0h^1)d^0=0$ and $g^i=h^id^{i-1}+d^ih^{i+1}$ for all $i=1,\cdots, n-1$.  Now suppose that $\cpx{g}$ is in $\Ker\varphi$, and let $g: Y\oplus M\ra Y\oplus M$ be in ${\cal C}$ induced by that above commutative diagram, that is, $\theta(\cpx{g})=\bar{g}$. Then $\tilde{d}^ng=g^n\tilde{d}^n=h^nd^{n-1}\tilde{d}^n=0$, and consequently $g\in\gh_M$ by Lemma \ref{lemma-ann} (1), and $\bar{g}=0$. Hence $\Ker\varphi\subseteq \Ker\theta$.  Conversely, suppose that $\cpx{g}$ is a chain map in $\End_{\Cb{{\cal C}}}(\cpx{T})$ such that $\theta(\cpx{g})=\bar{g}=0$.  Then $g^n\tilde{d}^n=\tilde{d}^ng=0$. By Definition \ref{def-sym-approx-sequence} (1), there is a map $h^n: T^n\lra T^{n-1}$ in ${\cal C}$ such that $h^nd^{n-1}=g^n$.  Now $(g^{n-1}-d^{n-1}h^n)d^{n-1}=g^{n-1}d^{n-1}-d^{n-1}g^n=0$. If $n>1$, then by Definition \ref{def-sym-approx-sequence} (1) again, we get a morphism $h^{n-1}:T^{n-1}\lra T^{n-2}$ such that $g^{n-1}-d^{n-1}h^n=h^{n-1}d^{n-2}$.  Using Definition \ref{def-sym-approx-sequence} (1) repeatedly, we get morphisms $h^i: T^i\lra T^{i-1}$ for $i=1, \cdots, n$ such that $g^n=h^nd^{n-1}$ and $g^i=h^id^{i-1}+d^ih^{i+1}$ for all $i=1, \cdots, n$.  Finally, $(g^0-d^0h^1)d^0=g^0d^0-d^0h^1d^0=d^0g^1-d^0(g^1-d^1h^2)=0$.  Hence $\cpx{g}$ is in $\Ker\varphi$, and consequently $\Ker\theta\subseteq\Ker\varphi$.

Altogether, we have shown that $\theta$ and $\varphi$ are surjective ring homomorphisms with the same kernel. Hence $\End_{{\cal C}/\gh_M}(Y\oplus M)$ and $\End_{\Kb{{\cal C}/\cogh_M}}(\cpx{T})$ are isomorphic, and the theorem is proved.
\end{proof}

\begin{Koro}
Let $\mathcal{C}$ be an additive category, and let $M$ be an object in $\mathcal{C}$. Suppose that $$X\lra M_1\lra\cdots\lra M_n\lra Y$$ is a higher $\add(M)$-split sequence. Then $\End_{\mathcal{C}}(M\oplus X)$ and $\End_{\mathcal{C}}(M\oplus Y)$ are derived equivalent.
\label{corollary-higher-D-split-seq}
\end{Koro}

 \medskip

 Let $A$ be a finite dimensional algebra, and let $P$ be a projective $A$-module with $\nu_AP\simeq P$, where $\nu_A$ is the Nakayama functor $D\Hom_A(-, A)$. Suppose that $Y$ is an $A$-module admitting an $\add(P)$-presentation, that is, there is an exact sequence $P_1\raf{f_1} P_0\raf{f_0} Y\ra 0$ in $\modcat{A}$ with $P_i\in\add(P)$ for $i=0, 1$.  Let $P_2\ra \Ker f_1$ be a right $\add(P)$-approximation of $\Ker f_1$, we get a sequence $P_2\raf{f_2} P_1\raf{f_1} P_0\raf{f_0} Y$.  Continuing this process by taking a right $\add(P)$-approximation  $P_i\ra \Ker f_{i-1}$ for $2\leq i\leq n$, we get a complex
$$(\ddagger)\quad \quad 0\lra X\lraf{f_{n+1}} P_n\lraf{f_n}P_{n-1}\lra \cdots\lra P_1\lraf{f_1}P_0\lraf{f_0}Y\lra 0, $$
 where $f_{n+1}$ is the kernel of $f_n$.
 \begin{Koro}
In the above sequence $(\ddagger)$, the algebras $\End_A(P\oplus X)$ and $\End_{A}(P\oplus Y)$ are derived equivalent.
 \label{corollary-nustable}
 \end{Koro}
 \begin{proof}
By definition, the sequence $P_n\raf{f_n}\cdots\ra P_0\raf{f_0}Y$  is a right $\add(P)$-approximation sequence. The assumption $P\simeq\nu_AP$, together with the natural isomorphism $D\Hom_A(P, -)\simeq\Hom_A(-, \nu_AP)$, implies that  $X\raf{f_{n+1}} P_n\raf{f_n} \cdots\raf{f_1}P_0$ is a left $\add(P)$-approximation.  Together with the fact that  $f_0$ is a cokernel of $f_1$ and $f_{n+1}$ is a kernel of $f_n$, we see that the sequence $(\ddagger)$ is a higher $\add(P)$-split sequence. The corollary then follows from Corollary \ref{corollary-higher-D-split-seq}.
\end{proof}

 Corollary \ref{corollary-nustable} provides an easy construction of derived equivalences, as illustrated by the following example.

 \medskip
{\noindent\bf Example}. Let $A$ be the Nakayama algebra give by the quiver
$$
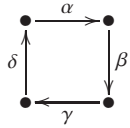
\xymatrix{
\bullet \ar[r]^{\alpha} &\bullet \ar[d]^{\beta}\\
\bullet \ar[u]^{\delta} &\bullet \ar[l]^{\gamma}\\
}$$
with relations $\alpha\beta\gamma\delta\alpha=\beta\gamma\delta\alpha\beta=\gamma\delta\alpha\beta\gamma=\delta\alpha\beta\gamma\delta=0$. We denote by $P_i$ the indecomposable projective $A$-module corresponding to the vertex $i$. Let $P=P_1\oplus P_3$, and let $Y$ be the module $\begin{smallmatrix}1\\2\end{smallmatrix}$, which admits an $\add(P)$-presentation $P_3\ra P_1\ra Y\ra 0$.  Using the method in Corollary \ref{corollary-nustable}, we can construct a sequence
$$0\lra X\lra P_3\lraf{f_2} P_3\lraf{f_1} P_1\lraf{f_0} Y\lra 0,$$
 where $X$ is the is the module
   $\begin{smallmatrix} 4\\1\\2\\3\end{smallmatrix}$. Note that the above sequence is not exact at the right $P_3$.  Using Corollary \ref{corollary-nustable}, we can deduce that $\End_A(P_1\oplus P_3\oplus Y)$ and $\End_A(P_1\oplus P_3\oplus X)$ are derived equivalent.

\section{Symmetric approximation sequences in weakly $n$-angulated categories}

Theorem \ref{theorem-ghosts} tells us that symmetric approximations sequences in arbitrary additive categories give rise to derived equivalences between quotient rings of endomorphism rings modulo ghost ideals and coghost ideals.  In this section, we will see that, if the category $\mathcal{C}$ in Theorem \ref{theorem-ghosts} has some ``weak'' $n$-angulated structure and the symmetric approximation sequence is an $n$-angle in $\mathcal{C}$, then the ideals can be chosen to be ideals of factorizable ghosts and coghosts respectively.

The notion of $n$-angulated category is given in \cite{Geiss2013} as a generalization of triangulated categories (In this case $n=3$). Typical examples of $n$-angulated categories include certain $(n-2)$-cluster tilting subcategories in a triangulated category, which appear in recent cluster tilting theory. The ``weak'' $n$-angulated structure we need in this section is obtained from the definition of $n$-angulated categories \cite{Geiss2013} by dropping some axioms.

\begin{Def}
Let $n\geq 3$ be an integer.
A {\bf weakly $n$-angulated category} is an additive category ${\cal C}$ together with an automorphism $\Sigma$ of ${\cal C}$, and a class $\pentagon$ of $n$-angles of the form
$$X_1\lraf{f_1} X_2\lraf{f_2}\cdots\lraf{f_{n-1}} X_n\lraf{f_n} \Sigma X_1$$
satisfying the following axioms:

${\rm (F1')}$. For each $X\in {\cal C}$, the sequence $X\lraf{1_X}X\lra 0\lra\cdots\lra 0\lra \Sigma X$  belongs to $\pentagon$. The class $\pentagon$ is closed under taking direct sums and direct summands, and is closed under isomorphisms.

${\rm (F2)}$. $X_1\raf{f_1} X_2\raf{f_2}\cdots\raf{f_{n-1}}X_n\raf{f_n}\Sigma X_1$ is in $\pentagon$ if and only if so is $X_2\raf{f_2}\cdots\raf{f_{n-1}} X_n\raf{f_n}\Sigma X_1\lraf{(-1)^n\Sigma f_1}\Sigma X_2$.

${\rm (F3)}$.  For each commutative diagram
$$\xymatrix@M=1mm{
X_1\ar[r]^{f_1}\ar[d]^{h_1} &X_2\ar[r]^{f_2}\ar[d]^{h_2} & X_3\ar[r]^{f_3}\ar@{..>}[d]^{h_3} &\cdots\ar[r]^{f_{n-1}} &X_n\ar@{..>}[d]^{h_n}\ar[r]^{f_n} &\Sigma X_1\ar[d]^{\Sigma h_1}\\
Y_1\ar[r]^{g_1} &Y_2\ar[r]^{g_2} &Y_3\ar[r]^{g_3} &\cdots\ar[r]^{g_{n-1}} &Y_n\ar[r]^{g_n} &\Sigma Y_1\\
}$$
with rows in $\pentagon$ can be completed to a morphism of $n$-angles.
\label{Definition-weakly-n-angle}
\end{Def}

The axioms in Definition \ref{Definition-weakly-n-angle} are obtained from the axioms (F1), (F2), (F3) and (F4) in the definition of   $n$-angulated categories in \cite{Geiss2013} by dropping the pushout axiom (F4) and by dropping the condition (F1)(c)  (each morphism can be embedded into an $n$-angle) to obtain (F1$'$).

\medskip
{\it\noindent Remark.} (a)
The relationship between weakly $n$-angulated categories and $n$-angulated categories is like the relationship between additive categories and abelian categories. In an abelian category,  pullback and pushout always exist, and every morphism has a kernel and a cokernel, while additive categories do not have these properties in general. Correspondingly,  an $n$-angulated category has a pushout axiom (F4), and every morphism can be embedded into an $n$-angle. However, a weakly $n$-angulated category does not necessarily have these properties.

 (b). Just like additive categories, the axioms of Definition \ref{Definition-weakly-n-angle} can be easily satisfied by many full subcategories of $n$-angulated categories.  Suppose that  $({\cal C}, \pentagon)$ is a weakly $n$-angulated category, and that ${\cal C}'$ is a full additive  subcategory of ${\cal C}$  such that $\Sigma C'=C'$. Denote by $\pentagon'$ the intersection $\pentagon\cap {\cal C}'$. Then it is easy to see that $({\cal C}', \pentagon')$ is again a weakly $n$-angulated category.  In particular, every full additive  subcategory of an $n$-angulated category closed under $\Sigma$ is weakly $n$-angulated.

\medskip
An additive covariant functor $H$ from a weakly $n$-angulated category $({\cal C}, \pentagon)$ to $\Modcat{\mathbb{Z}}$ is called {\em  cohomological}, if whenever $$X_1\lraf{f_1} X_2\lraf{f_2}\cdots\lraf{f_{n-1}} X_n\lraf{f_n} \Sigma X_1$$
is an $n$-angle in $\pentagon$, the long sequence
$$\cdots\lra H(\Sigma^iX_1)\lraf{H(\Sigma^if_1)}H(\Sigma^iX_2)\lraf{H(\Sigma^if_2)}\cdots\lraf{H(\Sigma^if_{n-1})} H(\Sigma^iX_n)\lraf{H(\Sigma^if_n)}H(\Sigma^{i+1}X_1)\lra\cdots$$
is exact.  Dually we have {\em contravariant cohomological functors}.

%A contravariant additive functor $H$ from  $({\cal C}, \pentagon)$ to $\Modcat{\mathbb{Z}}$ is called {\em (contravariantly) cohomological}, if whenever
%$$X_1\lraf{f_1} X_2\lraf{f_2}\cdots\lraf{f_{n-1}} X_n\lraf{f_n} \Sigma X_1$$
%is an $n$-angle in $\pentagon$, the long sequence
%$$\cdots\lra H(\Sigma^{i+1}X_1)\lraf{H(\Sigma^if_n)}H(\Sigma^iX_n)\lraf{H(\Sigma^if_{n-1})}\cdots\lraf{H(\Sigma^if_{2})} H(\Sigma^iX_2)\lraf{H(\Sigma^if_1)}H(\Sigma^{i}X_1)\lra\cdots$$
%is exact.

\begin{Lem}
  Let ${\cal C}$ be a weakly $n$-angulated category, and let
  $$X_1\lraf{f_1} X_2\lraf{f_2}\cdots\lraf{f_{n-1}} X_n\lraf{f_n} \Sigma X_1$$
  be an $n$-angle in ${\cal C}$. Then we have the following:

  $(1)$. $f_if_{i+1}=0$ for all $i=1,2,\cdots, n-1$;

  $(2)$.  ${\cal C}(X, -)$ and ${\cal C}(-, X)$ are cohomological;

  $(3)$. Suppose that $2\leq i< n$. Each commutative diagram
 $$\xymatrix@M=1mm{
X_1\ar[r]^{f_1}\ar[d]^{h_1} &X_2\ar[r]^{f_2}\ar[d]^{h_2} &\cdots\ar[r]^{f_{i-1}}& X_i\ar[r]^{f_i}\ar[d]^{h_i}&X_{i+1}\ar[r]^{f_{i+1}}\ar@{..>}[d]^{h_{i+1}} &\cdots\ar[r]^{f_{n-1}} &X_n\ar@{..>}[d]^{h_n}\ar[r]^{f_n} &\Sigma X_1\ar[d]^{\Sigma h_1}\\
Y_1\ar[r]^{f_1} &Y_2\ar[r]^{g_2} &\cdots\ar[r]^{g_{i-1}}&Y_i\ar[r]^{g_i} &Y_{i+1}\ar[r]^{g_{i+1}}&\cdots\ar[r]^{g_{n-1}} &Y_n\ar[r]^{g_n} &\Sigma Y_1\\
}$$
with rows in $\pentagon$ can be completed in ${\cal C}$ to a morphism of $n$-angles
\label{Lemma-n-angle}
  \end{Lem}

\begin{proof}
The proof is almost the same as in \cite{Geiss2013}.
 \end{proof}

\begin{Theo}
  Let $({\cal C}, \Sigma)$ be a weakly $n$-angulated category ($n\geq 3$), and let $M$ be an object in ${\cal C}$. Suppose that
  $$X\lraf{f}M_1\lra\cdots\lra M_{n-2}\lraf{g} Y\lraf{\eta} \Sigma X$$
 is an $n$-angle in ${\cal C}$ such that $f$ and $g$ are left and right $\add(M)$-approximations, respectively. Then the quotient rings
 $$\displaystyle{\frac{\End_{\cal C}(M\oplus X)}{\fcogh_M(M\oplus X)}}
\quad \mbox{ and }\quad \displaystyle{\frac{\End_{\mathcal{C}}(M\oplus Y)}{\fgh_M(M\oplus Y)}}$$
 are derived equivalent.
 \label{theorem-Fghosts}
\end{Theo}

 \begin{proof}
The proof is similar to that of Theorem \ref{theorem-ghosts}.
By  Definition \ref{Definition-weakly-n-angle}, the sequence
$$X\lraf{f}M_1\lra\cdots\lra M_{n-2}\oplus M\lraf{\tilde{g}} Y\oplus M\lraf{\tilde{\eta}} \Sigma X,$$
where $\tilde{g}=\left[\begin{smallmatrix}
g&0\\0&1\end{smallmatrix}
\right]$ and $\tilde{\eta}:=\left[\begin{smallmatrix}\eta\\0 \end{smallmatrix}\right]$, is still an $n$-angle.  Moreover, the morphism $\tilde{g}$ is still a right $\add(M)$-approximation. Thus, by our assumptions together with Lemma \ref{Lemma-n-angle}, it is easy to check that
$$X\lraf{f}M_1\lra\cdots\lra M_{n-2}\oplus M\lraf{\tilde{g}}Y\oplus M$$
is a symmetric $\add(M)$-approximation sequence.

 We denote by $\cpx{T}$ be the complex
$$0\lra X\lraf{f}M_1\lra\cdots\lra M_n\oplus M\lra 0$$
with $X$ in degree zero.  Then by Corollary \ref{corollary-sym-seq-tilting}, the complex $\cpx{T}$ is self-orthogonal in $\Kb{\mathcal{C}/\fcogh_M}$.
As we have done similarly in the proof of Theorem \ref{theorem-ghosts}, it is easy to prove that ${\cpx{\tilde{T}}}:=\HomP_{\mathcal{C}/\fcogh_M}(M\oplus X, \cpx{T})$  is a tilting complex over $\End_{\mathcal{C}/\fcogh_M}(M\oplus X)$. It remains to show that the endomorphism ring of $\cpx{\tilde{T}}$, which is isomorphic to $\End_{\mathcal{C}/\fcogh_M}(\cpx{T})$, is isomorphic to $\End_{\mathcal{C}/\fgh_M}(Y\oplus M)$.

 \medskip
Firstly, for each chain map $\cpx{u}$ in $\End_{\Cb{\cal C}}(\cpx{T})$, by Lemma \ref{Lemma-n-angle} (3),  there is a morphism $u\in\End_{\cal C}(Y\oplus M)$ such that the diagram
$$\quad \xymatrix{
T^0\ar[r]^{d_T^0}\ar[d]^{u^0}   &T^1\ar[r]\ar[r]^{d_T^1}\ar[d]^{u^1} & \cdots\ar[r]^{d_T^{n-1}} &T^{n}\ar[d]^{u^n}\ar[r]^{\tilde{g}}& Y\oplus M\ar[r]^{\tilde{\eta}}\ar@{-->}[d]^{u} &\Sigma T^0\ar[d]^{\Sigma u^0\quad\quad(\bigstar)}\\
T^0\ar[r]^{d_T^0} &T^1\ar[r]^{d_T^1} & \cdots\ar[r]^{d_T^{n-1}} &T^{n}\ar[r]^{\tilde{g}}& Y\oplus M\ar[r]^{\tilde{\eta}} &\Sigma T^0\\
}$$
is commutative. If $u'$ is another morphism in $\End_{\cal C}(Y\oplus M)$ making the above diagram commutative, then $\tilde{g}(u-u')=0=(u-u')\tilde{\eta}$.  Since $\tilde{g}$ is a right $\add(M)$-approximation by our assumption, the morphism $(u-u')$ belongs to $\gh_M(Y\oplus M)$ by Lemma \ref{lemma-ann} (1). It follows from $(u-u')\tilde{\eta}=0$ that $u-u'$ factorizes through $T^n$, which is in $\add(M)$. Hence $u-u'$ is in $\fgh_M(Y\oplus M)$. Denote by $\bar{u}$ the morphism in ${\cal C}/\fgh_M$ corresponding to $u$. Thus, we get a map
$$\theta: \End_{\Cb{\cal C}}(\cpx{T})\lra \End_{{\cal C}/\fgh_M}(Y\oplus M)$$
sending $\cpx{f}$ to $\bar{u}$., which is clearly a ring homomorphism.  For each $u\in\End_{\cal C}(Y\oplus M)$, since $\tilde{g}$ is a right $\add(M)$-approximation, there is $u^n: T^n\lra T^n$ such that $\tilde{g}u=u^n\tilde{g}$. Thus, by the axioms (2) and (3) in Definition \ref{Definition-weakly-n-angle}, we get morphisms $u^i: T^i\lra T^i, i=0, \cdots, n$, making the above diagram commutative. This shows that $\theta$ is a surjective ring homomorphism.

Secondly, similar to the proof of Theorem \ref{theorem-ghosts}, one can prove that there is a surjective ring homomorphism
$$\varphi: \End_{\Cb{\cal C}}(\cpx{T})\lra \End_{\Kb{{\cal C}/\fcogh_M}}(\cpx{T}), $$
which is  the composite of the ring homomorphism
$$\End_{\Cb{\cal C}}(\cpx{T})\ra\End_{\Cb{{\cal C}/\fcogh_{\cal D}}}(\cpx{T})$$
 induced by the canonical functor ${\cal C}\ra {\cal C}/\fcogh_M$ and the canonical surjective ring homomorphism from $$\End_{\Cb{{\cal C}/\fcogh_M}}(\cpx{T})\lra\End_{\Kb{{\cal C}/\fcogh_M}}(\cpx{T}).$$

We have to show that $\theta$ and $\varphi$ have the same kernel. A chain map $\cpx{u}$ is in $\Ker\varphi$ if and only if there exist $h^i: T^i\ra T^{i-1}, i=1,\cdots, n$ in ${\cal C}$ such that  $u^0-d_T^0h^1$, $u^i-h^id_T^{i-1}-d_T^ih^{i+1}, i=1,\cdots, n-1$, and $u^n-h^nd_T^{n-1}$ are all in $\fcogh_M$.  Using the fact that $T^i\in\add(M)$ for all $i>0$ and that $d_T^0=f$ is a left $\add(M)$-approximation of $X$,  one can see, by Lemma \ref{lemma-ann}, that this is equivalent to saying that $u^n-h^nd_T^{n-1}=0$,  $u^i=h^id_T^{i-1}+d_T^ih^{i+1}$ for $i=1,\cdots, n-1$, and $u^0-d_T^0h^1\in\fcogh_M$.

 Let $\cpx{u}$ be in $\Ker\varphi$, and that  $u\in\End_{\cal C}(Y\oplus M)$ fits the commutative diagram $(\bigstar)$ above. Then $\theta(\cpx{u})=\bar{u}$.
We have $u^n=h^nd_T^{n-1}$, and consequently $\tilde{g}u=u^n\tilde{g}=h^nd_T^{n-1}\tilde{g}$, which is zero by Lemma \ref{Lemma-n-angle} (1).  It follows from Lemma \ref{lemma-ann} (1) that $u\in\gh_M$. The fact $\cpx{u}\in\Ker\varphi$ also implies that $u^0-d_T^0h^1\in\fcogh_M$.  In particular, the morphism $u^0-d_T^0h^1$ factorizes through an object in $\add(M)$.  Assume that $u^0-d_T^0h^1=ab$ for some $a\in{\cal C}(T^0, M')$ and $b\in{\cal C}(M', T^0)$ with $M'\in\add(M)$. Since $d_T^0$ is a left $\add(M)$-approximation, we see that $a$ factorizes through $d_T^0$, and hence $u^0-d_T^0h^1$ factorizes through $d_T^0$. Consequently, the morphism $u^0$ also factorizes through $d_T^0$, say, $u^0=d_T^0\alpha$.  Thus $\tilde{\eta}(\Sigma u^0)=\tilde{\eta}(\Sigma d_T^0)(\Sigma\alpha)$, which must be zero by the axiom (F2) in Definition \ref{Definition-weakly-n-angle} and  Lemma \ref{Lemma-n-angle} (1). Hence $u\tilde{\eta}=\tilde{\eta} (\Sigma u^0)=0$, and consequently $u$ factorizes through $T^n\in\add(M)$ by Lemma \ref{Lemma-n-angle} (2).  Altogether, we have shown that  $u$ belongs to $\fgh_M$. It follows that $\bar{u}=0$ and $\cpx{u}\in\Ker\theta$.  Hence $\Ker\varphi\subseteq\Ker\theta$.

  Conversely, suppose that $\cpx{u}\in\Ker\theta$ and $u\in\End_{\cal C}(Y\oplus M)$ fits the commutative diagram $(\bigstar)$.  Then $\theta(\cpx{u})=\bar{u}=0$, that is, $u\in\fgh_M(Y\oplus M)$.  Since $\tilde{g}$ is a right ${\cal D}$-approximation, by Lemma \ref{lemma-ann} (1), we have $\tilde{g}u=0$. Thus $u^n\tilde{g}=0$. By Lemma \ref{Lemma-n-angle} (2), there is a morphism $h^n: T^n\ra T^{n-1}$ such that $u^n=h^nd_T^{n-1}$. Now $(u^{n-1}-d_T^{n-1}h^n)d_T^{n-1}=u^{n-1}d_T^{n-1}-d_T^{n-1}u^n=0$. If $n\geq 2$, then, by Lemma \ref{Lemma-n-angle} (2), there is a morphism $h^{n-1}: T^{n-1}\ra T^{n-2}$ such that $u^{n-1}-d_T^{n-1}h^n=h^{n-1}d_T^{n-2}$. Moreover, $(u^{n-2}-d_T^{n-2}h^{n-1})d_T^{n-2}=d_T^{n-2}u^{n-1}-d_T^{n-2}h^{n-1}d_T^{n-2}=d_T^{n-2}d_T^{n-1}h^n=0$.  Repeating this process, we get $h^i: T^i\ra T^{i-1}, i=1,\cdots, n$ such that $u^n=h^nd_T^{n-2}$, $u^i=h^id_T^{i-1}+d_T^ih^{i+1}$ for $i=1,\cdots, n-1$, and $(u^0-d_T^0h^1)d_T^0=0$. Since $d_T^0$ is a left $\add(M)$-approximation, we deduce from Lemma \ref{lemma-ann} (2) that $u^0-d_T^0h^1\in\cogh_M(T^0)$.  Since $u$ factorizes through an object in $\add(M)$ and $\tilde{g}$ is a right $\add(M)$-approximation, it is easy to see that $u$ factorizes through $\tilde{g}$, and thus $\tilde{\eta}(\Sigma u^0)=u\tilde{\eta}=0$.
By  Lemma \ref{Lemma-n-angle} (2) and axiom (F2) in Definition \ref{Definition-weakly-n-angle}, the morphism $\Sigma u^0$ factorizes through $\Sigma d_T^0$, or equivalently,  $u^0$ factorizes through $d_T^0$. Hence  $u^0-d_T^0h^1$ factorizes through an object $\add(M)$, and consequently belongs to $\fcogh_M$. Thus we have shown that $\cpx{u}\in\Ker\varphi$, and  $\Ker\theta\subseteq\Ker\varphi$.

Thus $\theta$ and $\varphi$ have the same kernel, and the rings $\End_{\Kb{{\cal C}/\fcogh_M}}(\cpx{T})$ and $\End_{{\cal C}/\fgh_M}(Y\oplus M)$ are isomorphic, and the theorem then follows.
\end{proof}

Let $({\cal T}, \pentagon)$ and  $({\cal T}', \pentagon')$ be weakly $n$-angulated categories. An additive functor $F$ from ${\cal T}$ to ${\cal T}'$ is called an {\em $n$-angle functor} if  there is a natural isomorphism $\psi: \Sigma'F\ra F\Sigma $ and
$$F(X_1)\lraf{F(f_1)}F(X_2)\lraf{F(f_2)}\cdots\lraf{F(f_{n-1})}F(X_n)\lraf{F(f_n)\psi^{-1}_{X_1}}\Sigma'F(X_1)$$
is in $\pentagon'$ whenever
$$X_1\lraf{f_1}X_2\lraf{f_2}\cdots\lraf{f_{n-1}}X_n\lraf{f_n}\Sigma X_1$$
is in $\pentagon$.

\medskip
Now let $({\cal T}, \pentagon)$ be a weakly $n$-angulated category, and let $F$ be an $n$-angle functor from ${\cal T}$ to itself.  Suppose that $\Phi$ is an admissible subset of $\mathbb{Z}$, and $\Oc{\cal T}{F}{\Phi}$ is the $\Phi$-orbit category of ${\cal T}$. Then one may ask whether the $\Phi$-orbit category is again  naturally weakly $n$-angulated.  The answer is yes, as we shall prove in the following.

We fix a natural isomorphism $\psi(1): \Sigma F\lra F\Sigma$, and set $\psi(0):=id_{\Sigma}: \Sigma\ra \Sigma$.   For each positive integer $u$, we define $\psi(u): \Sigma F^u\lra F^u\Sigma$ to be the composite
$$\Sigma F^u\lraf{\psi(1){F^{u-1}}}F\Sigma F^{u-1}\lraf{F\psi(1){F^{u-2}}}F^2\Sigma F^{n-2}\lra\cdots\lraf{F^{u-1}\psi(1)} F^u\Sigma.$$
If $F$ is not an equivalence, then $F^{-1}=0$, and we define $\psi(u): \Sigma F^u\lra F^u\Sigma$ to be zero for all negative integers.  If $F$ is an equivalence, and $F^{-1}$ is a quasi-inverse of $F$, then $(F^{-1}, F)$ is an adjoint pair. Let $\epsilon: id_{\cal T}\lra F^{-1}F$ be the unit and let $\eta: FF^{-1}\lra id_{\cal T}$ be the counit. We define $\psi(-1)$ to be the composite
$$\Sigma F^{-1}\lraf{\epsilon {\Sigma F^{-1}}}F^{-1}F\Sigma F^{-1}\lraf{F^{-1}\psi(1)^{-1}{F^{-1}}}F^{-1}\Sigma FF^{-1}\lraf{F^{-1}\Sigma\eta}F^{-1}\Sigma, $$
and define $\psi(u)$, for each integer $u<0$, to be the composite
$$\Sigma F^{u}\lraf{\psi(-1){F^{u+1}}}F^{-1}\Sigma F^{u+1}\lraf{F^{-1}\psi(-1){F^{u+2}}}F^{-2}\Sigma F^{u+2}\lra\cdots\lra F^u\Sigma.$$
With these natural transformations in hand, we can define an automorphism $\Sigma^{\Phi}$ of $\Oc{\cal T}{F}{\Phi}$ as follows. $\Sigma^{\Phi} X$ is just $\Sigma X$ for each object $X$. For each homogeneous morphism $f_u: X\lra F^uY$ in $\Oc{\cal T}{F}{\Phi}$, we define $\Sigma^{\Phi}(f_u)$ to be the composite
$$\Sigma X\lraf{\Sigma f_u}\Sigma F^uY\lraf{\psi(u)_Y}F^u\Sigma Y.$$
One can check that $\Sigma^{\Phi}$ is indeed an automorphism of the $\Phi$-orbit category $\Oc{\cal T}{F}{\Phi}$.  Let $\pentagon^{\Phi}$ be the sequences in $\Oc{\cal T}{F}{\Phi}$ isomorphic to those $n$-angles in $\pentagon$.

\begin{Prop}
Keeping the notations above, the $\Phi$-orbit category $\Oc{\cal T}{F}{\Phi}$, together with $\Sigma^\Phi$ and $\pentagon^{\Phi}$, is a weakly $n$-angulated category.
\label{prop-phi-orbit-n-angle}
\end{Prop}
\begin{proof}
The axioms (F1$'$) and (F2) of Definition \ref{Definition-weakly-n-angle} are satisfied by the definition of $\pentagon^{\Phi}$ and $\Sigma^{\Phi}$.  Now given a commutative diagram
$$\xymatrix{
X_1\ar[r]^{f_1}\ar[d]^{h^1} &X_2\ar[r]^{f_2}\ar[d]^{h^2}&X_3\ar[r]^{f_3} &\cdots\ar[r]^{f_{n-1}} &X_n\ar[r]^{f_n} &\Sigma^{\Phi} X_1\ar[d]^{\Sigma^{\Phi} (h^1)}\\
Y_1\ar[r]^{g_1} &Y_2\ar[r]^{g_2} &Y_3\ar[r]^{g_3} &\cdots\ar[r]^{g_{n-1}} &Y_n\ar[r]^{g_n} &\Sigma^{\Phi}Y_1\\
}$$
in $\Oc{\cal T}{F}{\Phi}$ with rows in $\pentagon^{\Phi}$.  Clearly, we can assume that the rows are in $\pentagon$, and all the morphisms $f_i$ are homogeneous morphisms of degree zero for all $i=1, \cdots, n$.  Let $h^1=(h^1_u)_{u\in\Phi}$ and $h^2=(h^2_u)_{u\in\Phi}$. Then $h^1_ug_1=f_1h^2_u$ for all $u\in\Phi$. Thus, we get a commutative diagram
$$\xymatrix@M=1mm@C=12mm{
X_1\ar[r]^{f_1}\ar[d]^{h^1_u} &X_2\ar[r]^{f_2}\ar[d]^{h^2_u}&X_3\ar[r]^{f_3} &\cdots\ar[r]^{f_{n-1}} &X_n\ar[r]^{f_n} &\Sigma X_1\ar[d]^{\Sigma(h^1_u)}\\
F^u(Y_1)\ar[r]^{F^u(g_1)} &F^u(Y_2)\ar[r]^{F^u(g_2)} &F^u(Y_3)\ar[r]^{F^u(g_3)} &\cdots\ar[r]^{F^u(g_{n-1})} &F^u(Y_n)\ar[r]^{F^u(g_n)\psi(u)^{-1}} & \Sigma F^u(Y_1)\\
}$$
in ${\cal T}$ with rows in $\pentagon$.  Thus, in the weakly $n$-angulated category ${\cal T}$, we get a commutative diagram
$$\xymatrix@C=12mm{
X_1\ar[r]^{f_1}\ar[d]^{h^1_u} &X_2\ar[r]^{f_2}\ar[d]^{h^2_u}&X_3\ar[r]^{f_3}\ar@{-->}[d]^{h^3_u} &\cdots\ar[r]^{f_{n-1}} &X_n\ar[r]^{f_n}\ar@{-->}[d]^{h^n_u} &\Sigma X_1\ar[d]^{\Sigma(h^1_u)}\\
F^u(Y_1)\ar[r]^{F^u(g_1)} &F^u(Y_2)\ar[r]^{F^u(g_2)} &F^u(Y_3)\ar[r]^{F^u(g_3)} &\cdots\ar[r]^{F^u(g_{n-1})} &F^u(Y_n)\ar[r]^{F^u(g_n)\psi(u)^{-1}} & \Sigma F^u(Y_1)\\
}$$
for all $u\in\Phi$. Defining $h^i:=(h^i_u)_{u\in\Phi}$, we obtain a commutative diagram
$$\xymatrix{
X_1\ar[r]^{f_1}\ar[d]^{h^1} &X_2\ar[r]^{f_2}\ar[d]^{h^2}&X_3\ar[r]^{f_3}\ar[d]^{h^3} &\cdots\ar[r]^{f_{n-1}} &X_n\ar[r]^{f_n}\ar[d]^{h^n} &\Sigma^{\Phi} X_1\ar[d]^{\Sigma^{\Phi} (h^1)}\\
Y_1\ar[r]^{g_1} &Y_2\ar[r]^{g_2} &Y_3\ar[r]^{g_3} &\cdots\ar[r]^{g_{n-1}} &Y_n\ar[r]^{g_n} &\Sigma^{\Phi}Y_1\\
}$$
in $\Oc{\cal T}{F}{\Phi}$.  Thus $\Oc{\cal T}{F}{\Phi}$ satisfies the axiom (F3) of Definition \ref{Definition-weakly-n-angle}. Hence $\Oc{\cal T}{F}{\Phi}$, together with $\Sigma^{\Phi}$ and $\pentagon^{\Phi}$ is a weakly $n$-angulated category.
\end{proof}

\noindent
{\it Remark}.  For an   $n$-angulated category $\mathcal{T}$, the $\Phi$-orbit category $\Oc{\cal T}{F}{\Phi}$ is not an $n$-angulated category in general.

\medskip

In \cite{Hu2013a} and \cite{Chen2013},  the authors start from an $n$-angle
$$X\lraf{f}M_1\lra M_2\lra\cdots\lra M_{n-2}\lraf{g} Y\lraf{w} \Sigma X$$
in an $n$-angulated category ${\cal T}$ with $M_i\in\add(M)$ for some $M\in {\cal T}$. Let $F$ be an  $n$-angulated auto-equivalence of ${\cal T}$. The main results \cite[Theorem 1.1]{Hu2013a} and \cite[Theorem 3.1]{Chen2013} state that, for each admissible subset $\Phi$ of $\mathbb{Z}$, there is a derived equivalence between the quotient rings $\E_{\cal T}^{F, \Phi}(M\oplus X)/I$ and $\E_{\cal T}^{F,\Phi}(M\oplus Y)/J$ of the $\Phi$-Yoneda algebras provided that $f$(respectively, $g$) is a left (respectively, right) $\add_{\Oc{{\cal T}}{F}{\Phi}}(M)$-approximation in the $\Phi$-orbit category $\Oc{\cal T}{F}{\Phi}$ and ${\cal T}(F^iM, X)=0$ (respectively, ${\cal T}(Y, F^iM)=0$) for all $0\neq i\in\Phi$.  The ideals $I$ and $J$ are defined as follows.
Set $\bar{w}:=\left[\begin{smallmatrix}g\\0\end{smallmatrix}\right]: Y\oplus M\ra \Sigma X$ and $\tilde{w}:=[w, 0]: Y\lra \Sigma X\oplus M$. Then
$$I:=\left\{(x_i)\in\E_{\cal T}^{F, \Phi}(X\oplus M)|x_i=0, \forall0\neq i\in\Phi, x_0\mbox{ factorizes through }\add(M)\mbox{ and } \Sigma^{-1}(\tilde{w})\right\}, $$
$$J:=\left\{(x_i)\in\E_{\cal T}^{F, \Phi}(Y\oplus M)|x_i=0, \forall 0\neq i\in\Phi, x_0\mbox{ factorizes through }\add(M) \mbox{ and } \bar{w}\right\}.$$
This looks quite artificially defined.  However, we shall see that the ideals $I$ and $J$ are actually factorizable coghosts and ghosts in the $\Phi$-orbit category, and then one can use Theorem \ref{theorem-Fghosts} and Proposition \ref{prop-phi-orbit-n-angle} to give alternative proofs of the results in \cite{Hu2013a,Chen2013}.  Here we consider a more general situation: ${\cal T}$ is a weakly $n$-angulated category, and $F$ is just an $n$-angle endo-functor of ${\cal T}$, not necessarily an auto-equivalence.

\begin{Lem}
Keep the notations above and set ${\cal D}:=\add_{\Oc{{\cal T}}{F}{\Phi}}(M)$. We have the following.

$(1)$. If $g$ is a right ${\cal D}$-approximation in $\Oc{\cal T}{F}{\Phi}$, and ${\cal T}(Y, F^iM)=0$ for all $0\neq i \in \Phi$,  then $J=\fgh_{\cal D}(Y\oplus M)$;

$(2)$.  If $f$ is a left ${\cal D}$-approximation in $\Oc{\cal T}{F}{\Phi}$, and ${\cal T}(M, F^iX)=0$ for all $0\neq i\in\Phi$, then $I=\fcogh_{\cal D}(X\oplus M)$.
\label{Lemma-idealsIJ}
\end{Lem}

\begin{proof}
(1).  By Lemma \ref{lemma-ann}, we have $\fgh_{\cal D}(M, Y\oplus M)=0$.  Now we consider $\fgh_{\cal D}(Y, Y\oplus M)$. Since ${g}$ is a right ${\cal D}$-approximation. By Lemma \ref{lemma-ann} (1), one has $\fgh_{\cal D}(Y, Y\oplus M)$ consists of those morphisms $(x_i)\in\E_{\cal T}^{F, \Phi}(Y, Y\oplus M)$ satisfying the conditions:

\smallskip
 (a). $gx_i=0$ for all $i\in\Phi$;

 (b). There is some $(y_i)\in \Oc{\cal T}{F}{\Phi}(Y, M_Y\oplus M)$ such that $y_i*\tilde{g}=x_i$ for all $i\in\Phi$.

\smallskip
{\parindent=0pt By} our assumption that ${\cal T}(Y, F^iM)=0$ for all $0\neq i\in\Phi$, the morphism $y_i$ in condition (b) above is zero for all $0\neq i\in\Phi$. Thus $x_i=0$ for all $0\neq i\in\Phi$, $gx_0=0$ and $x_0$ factorizes through $\add(M)$ in ${\cal T}$.
The proof of (2) is dual.
\end{proof}

Combining Theorem \ref{theorem-Fghosts}, Proposition \ref{prop-phi-orbit-n-angle} and Lemma \ref{Lemma-idealsIJ}, we get the following corollary.

\begin{Koro}
 Let $({\cal T}, \Sigma)$ be a weakly $n$-angulated category ($n\geq 3$) with an $n$-angle endo-functor $F$, and let $M$ be an object in ${\cal T}$. Suppose that $\Phi$ is an admissible subset of $\mathbb{Z}$.  Let $$X\lraf{f}M_1\lra\cdots\lra M_{n-2}\lraf{g}Y\lra\Sigma X$$
 be an $n$-angle in ${\cal T}$ such that $M_i\in\add(M)$ for all $i=1,\cdots,n-2$, and that $f$  and $g$ are   left and right $\add(M)$-approximations in the orbit category $\Oc{\cal T}{F}{\Phi}$,  respectively. Suppose that ${\cal T}(Y, F^iM)=0={\cal T}(M, F^iX)$ for all $0\neq i\in\Phi$.  Then the rings $\E_{\cal T}^{F, \Phi}(M\oplus X)/I$ and $\E_{\cal T}^{F,\Phi}(M\oplus Y)/J$ are derived equivalent.
\end{Koro}

This corollary generalizes the results \cite[Theorem 3.1]{Hu2013a} and \cite[Theorem 1.1]{Chen2013}:   the functor $F$ here is not necessarily an auto-equivalence,  while this is required in both \cite{Hu2013a} and \cite{Chen2013}.

\section{Examples}

In this section, we  give some examples to illustrate our main results.

Throughout this section, we assume that $A$ is a self-injective artin algebra, and write $${\cal D}_{n}:=\add\big(\bigoplus_{i=0}^n\Sigma^{-i}A\big),$$
in $\K{\modcat{A}}$ for $n\in \mathbb{N}$.
  For simplicity, we will write $\mathscr{K}$ for $\K{\modcat{A}}$.  It is straightforward to prove the following lemma.
\begin{Lem}
$\cogh_{{\cal D}_n}$ and $\gh_{{\cal D}_n}$ are equal, and both of them consist of morphisms  $\cpx{\alpha}$ such that $H^i(\cpx{\alpha})=0$ for all $0\leq i\leq n$.
\label{lemma-ann-ghost}
\end{Lem}

 {\em Remark. } This lemma implies $\fcogh_{{\cal D}_n}$ and $\fgh_{{\cal D}_n}$ also coincide. Moreover, if $\cpx{M}$ is a complex over $A$ with zero homology in all degrees not in $\{0, 1,\cdots, n\}$, then $\cogh_{{\cal D}_n}(\cpx{M})$ consists of {\em ghost maps}, that is, chain maps $\cpx{\alpha}: \cpx{M}\lra \cpx{M}$ such that $H^i(\cpx{\alpha})=0$ for all $i\in\mathbb{Z}$. We denote by $\mathcal{G}$ the ideal of $\mathscr{K}$ consisting of ghost maps. Then $\cogh_{{\cal D}_n}(\cpx{M})=\gh_{{\cal D}_n}(\cpx{M})=\mathcal{G}(\cpx{M})$. Let $\mathcal{G}_{\cal D}:=\mathcal{G}\cap \mathcal{F}_{\cal D}$. Then $\mathcal{G}_{\cal D}(\cpx{M})=\fcogh_{\mathcal{D}}(\cpx{M})=\fgh_{\cal D}(\cpx{M})$

\medskip
Let $\cpx{X}$ be a complex over $\modcat{A}$, and let $i$ be an integer.  Suppose that $\pi^i: P_X^i\lra H^i(\cpx{X})$ be a projective cover of the $i$-th homology of $\cpx{X}$.  Then $\pi^i$ can be lifted to a morphism $h^i: P_X^i\ra\Ker d_X^{i}$ along the canonical epimorphism $\Ker d_X^i\lra H^i(\cpx{X})$.  Let $f^i: P_X^i\lra X^i$ be the composite of $h^i$ and the inclusion $\Ker d_X^i\hookrightarrow X^i$. Then $f^id_X^i=0$, and $f^i$ gives rise to a chain map from $\Sigma^{-i}P_X^i\lra \cpx{X}$.  This can be illustrated in the following graph.
$$\xy<8mm, 0mm>:
(.5, 0)*+{X^{i-1}}="xi-1",
(4, 0)*+{X^i}="xi",
(7.5, 0)*+{X^{i+1}}="xi+1",
(2, -1)*+{\Ker d_X^i}="ki",
(5.5, -1)*+{H^i(\cpx{X})}="hi",
(4, 2.5)*+{P_X^i}="pi1",
(5.5, 1.5)*+{P_X^i}="pi2",
{\ar@{->>} "ki"; "hi"},
{\ar@{->>}_(.4){\pi^i} "pi2"; "hi"},
{\ar^{d_X^{i-1}} "xi-1"; "xi"},
{\ar^(.6){d_X^i} "xi"; "xi+1"},
{\ar@{^{(}->} "ki"; "xi"},
{\ar@{=} "pi1"; "pi2"},
{\ar_{f^i} "pi1"; "xi"},
{\ar@/_.5pc/@{-->}^(.3){h^i} "pi2"; "ki"},
\endxy$$
Define
$$\cpx{P_X}:=\coprod_{i\in\mathbb{Z}}\Sigma^{-i}P_X^i.$$
Then we get a chain map $\cpx{f}: \cpx{P_X}\lra \cpx{X}$.  Form a triangle
$$\cpx{Y}\lraf{\cpx{g}}\cpx{P_X}\lraf{\cpx{f}}\cpx{X}\lra\Sigma\cpx{Y} \quad\quad  (\star)$$
in $\mathscr{K}$ and set ${\cal D}_{\infty}:=\add\{\Sigma^{-i}A\mid i\in\mathbb{Z}\}$.  We claim that $(\star)$ is actually a ${\cal D}_{\infty}$-split triangle in $\mathscr{K}$.  Indeed, for each $i\in\mathbb{Z}$,  there is a commutative diagram
$$\xymatrix@M=1mm@C=20mm{
\mathscr{K}(\Sigma^{-i}A, \cpx{P_X})\ar[r]^{\mathscr{K}(\Sigma^{-i}A, \cpx{f})}\ar[d]^{\simeq} & \mathscr{K}(\Sigma^{-i}A, \cpx{X})\ar[d]^{\simeq}\\
P_X^i\ar@{->>}[r]^{\pi^i} & H^i(\cpx{X}).
}$$
Hence $\mathscr{K}(\Sigma^{-i}A, \cpx{f})$ is surjective for all $i\in\mathbb{Z}$.  It follows that $\mathscr{K}(\Sigma^{-i}A, \cpx{f})$ is surjective for all $i\in\mathbb{Z}$. Note that there is a natural isomorphism $\mathscr{K}(\cpx{P}, -)\simeq D\mathscr{K}(-, \nu_A\cpx{P})$ for a bounded complex $\cpx{P}$ of finitely generated projective $A$-modules, where $D$ is the usual duality of $A$-modules and $\nu_A$ is the Nakayama functor. It follows that $\mathscr{K}(\cpx{f}, \Sigma^{-i}(\nu_AA))$ is injective for all $i\in\mathbb{Z}$. Since $A$ is self-injective, we have $\add(\nu_AA)=\add({}_AA)$.  Hence $\mathscr{K}(\cpx{f}, \Sigma^{-i}A)$ is injective for all $i\in\mathbb{Z}$.  Using the long exact sequences obtained by applying $\mathscr{K}(A, -)$ and $\mathscr{K}(-, A)$ to the triangle $(\star)$, we deduce that the sequences
$$0\lra \mathscr{K}(\Sigma^{-i}A, \cpx{Y})\lra\mathscr{K}(\Sigma^{-i}A, \cpx{P_X})\lra \mathscr{K}(\Sigma^{-i}A, \cpx{X})\lra 0, \quad (*)$$
$$0\lra \mathscr{K}( \cpx{X}, \Sigma^{-i}A)\lra\mathscr{K}(\cpx{P_X}, \Sigma^{-i}A)\lra \mathscr{K}(\cpx{X}, \Sigma^{-i}A)\lra 0 \quad (**)$$
are exact for all $i\in\mathbb{Z}$. Particularly, $\cpx{f}$ is a right $\mathcal{D}_{\infty}$-approximation and $\cpx{g}$ is a left $\mathcal{D}_{\infty}$-approximation.   Note that the exact sequence $(*)$ is isomorphic to the sequence
$$0\lra H^i(\cpx{Y})\lra P_X^i\lra H^i(\cpx{X})\lra 0.$$
This shows that $H^i(\cpx{Y})=\Omega(H^i(\cpx{X}))$.

\medskip
Now assume that $\cpx{X}$ is a bounded complex
$$0\lra X^0\lraf{d^0} X^1\lraf{d^1}\cdots\lraf{d^{n-1}} X^n\lra 0.$$
Then $P_X^i=0$ for all $i\not\in\{0, \cdots, n\}$. Thus $\cpx{P_X}$ lies in $\mathcal{D}_n$, and the triangle $(\star)$ is a $\mathcal{D}_n$-split triangle.  By Theorem \ref{theorem-Fghosts} and the remark after Lemma \ref{lemma-ann-ghost}, the algebras $$\End_{\mathscr{K}/\mathcal{G}_{\cal D}}\big(\cpx{Y}\oplus\bigoplus_{i=0}^n\Sigma^{-i}A\big)\mbox{ and }\End_{\mathscr{K}/\mathcal{G}_{\cal D}}\big(\cpx{X}\oplus\bigoplus_{i=0}^n\Sigma^{-i}A\big)$$
are derived equivalent. Similarly,  by Theorem \ref{theorem-ghosts} and the remark after Lemma \ref{lemma-ann-ghost}, the algebras
$$\End_{\mathscr{K}/\mathcal{G}}\big(\cpx{Y}\oplus\bigoplus_{i=0}^n\Sigma^{-i}A\big)\mbox{ and }\End_{\mathscr{K}/\mathcal{G}}\big(\cpx{X}\oplus\bigoplus_{i=0}^n\Sigma^{-i}A\big)$$
are also derived equivalent.  In the following, we shall see that these algebras have a very nice form.

\begin{Lem} \label{ex1}
$H^i: \mathscr{K}(\cpx{X}, \Sigma^{-i}A)\lra \Hom_A(H^i(\cpx{X}), A)$ is an isomorphism in $\modcat{A\opp}$ for all $0\leq i\leq n$.
\end{Lem}
\begin{proof}
We first show that $H^i$ is monic. Suppose that $H^i(\cpx{f})=0$. Since $H^j(\Sigma^{-i}A)=0$ for all $j\neq i$, we have $H^j(\cpx{f})=0$ for all $j\in\mathbb{Z}$. It follows from Lemma \ref{lemma-ann-ghost}  that $\cpx{f}\in \cogh_{\cal D}(\cpx{X}, \Sigma^{-i}A)=0$.  Let $f: H^i(\cpx{X})\lra A$ be an $A$-module homomorphism. Since $A$ is self-injective, $f$ factorizes through the inclusion $H^i(\cpx{X})\hookrightarrow X^i/\Img d^{i-1}$. This gives rise to a chain map  $\cpx{f}: \cpx{X}\lra \Sigma^{-i}A$ such that $f^i: X^i\lra A$ is the composite $X^i\lra X^i/\Img d^{i-1}\lra A$. Notices that $H^i(\cpx{f})=f$. We conclude that $H^i$ is also epic.
\end{proof}

 With the preparations above, we can write

$$\End_{\mathscr{K}/\mathcal{G}}\big(\cpx{X}\oplus\bigoplus_{i=0}^n\Sigma^{-i}A\big)=\begin{pmatrix}
\End_{\mathscr{K}}(\cpx{X})/\mathcal{G}(\cpx{X})& (H_X^0)^*&\cdots & (H_X^n)^*\\
H_X^0 & A&\cdots&\ 0\\
\vdots &\vdots&\ddots&0\\
H_X^{n}& 0&\cdots & A
\end{pmatrix}, $$
$$\End_{\mathscr{K}/\mathcal{G}}\big(\cpx{Y}\oplus\bigoplus_{i=0}^n\Sigma^{-i}A\big)=\begin{pmatrix}
\End_{\mathscr{K}}(\cpx{Y})/\mathcal{G}(\cpx{Y})& (\Omega(H_X^0))^*&\cdots & (\Omega(H_X^n))^*\\
\Omega(H_X^0) & A&\cdots&\ 0\\
\vdots &\vdots&\ddots&0\\
\Omega(H_X^{n})& 0&\cdots & A
\end{pmatrix}, $$
The algebras $\End_{\mathscr{K}/\mathcal{G}_{\cal D}}\big(\cpx{X}\oplus\bigoplus_{i=0}^n\Sigma^{-i}A\big)$ and $\End_{\mathscr{K}/\mathcal{G}_{\cal D}}\big(\cpx{Y}\oplus\bigoplus_{i=0}^n\Sigma^{-i}A\big)$ have similar forms: just replacing $\mathcal{G}$ with $\mathcal{G}_{\cal D}$ in the above matrices.

\medskip

In the following, we will illustrate  our results by a concrete example.

{\bf Example. } Let $k$ be a field, and let $A=k[x]/(x^n)$. Suppose that $1\leq m\leq n-1$ and that  $\cpx{X}$ is the complex
$$0\lra A\lraf{\cdot x^m} A\lra 0$$
with the left $A$ in degree zero.  The endomorphism algebra $\End_{\K{{A}}/\mathcal{G}}(\cpx{X}\oplus A\oplus \Sigma^{-1}A)$ is denoted by $\Lambda(n, m)$. We first give some morphisms in $\Lambda(n, m)$.

$$\alpha_1: \xymatrix{0\ar[d]\ar[r] & A\ar[d]^{\cdot x}\\
0\ar[r] & A}\quad  \alpha_2:\xymatrix{A\ar[d]_{\cdot x}\ar[r] & 0\ar[d]\\
A\ar[r] & 0}, \quad\alpha_3:{}\xymatrix{A\ar[r]^{.x^m}\ar[d]^{.x}&A\ar[d]^{.x}\\
A\ar[r]^{.x^{m}}&A
}, $$

$$\beta_1:{}\xymatrix{0\ar[r]\ar[d]&A\ar[d]^{id}\\A\ar[r]^{.x^m}&A
}, \quad\beta_2:{}\xymatrix{A\ar[r]^{.x^m}\ar[d]&A\ar[d]^{.x^{n-m}}\\
0\ar[r]&A}, \quad \beta_3:{}\xymatrix{A\ar[r]^{.x^m}\ar[d]^{id}&A\ar[d]\\A\ar[r]&0}, \quad \beta_4:{}\xymatrix{A\ar[r]\ar[d]^{.x^{n-m}}&0\ar[d]\\A\ar[r]^{.x^m}&A}.$$

{\bf\parindent=0pt Case I: } $1<m<n-1$. In this case, the above morphisms are irreducible and the algebra $\Lambda(n, m)$ is given by the following quiver with relations.
$$\begin{array}{c}
\xymatrix{{\bullet}\ar@(lu,ld)_(.4){\alpha_1}\ar@<2pt>[r]^{\beta_1}&\bullet\ar@<2pt>[l]^{\beta_2}\ar@<2pt>[r]^{\beta_3}\ar@(ru, lu)_{\alpha_3}&\bullet\ar@<2pt>[l]^{\beta_4}\ar@(ru,rd)^{\alpha_2}\\
}\\
\alpha_3^m, \, \beta_1\beta_3,\,  \beta_4\beta_2,\,  \alpha_1^m\beta_1,\,  \beta_2\alpha_1^m, \alpha_2^m\beta_4,\,  {\beta_3}\alpha_2^m\\
 \alpha_1^{n-m}-\beta_1\beta_2, \,
  \alpha_2^{n-m}-\beta_4\beta_2, \,
    \alpha_3^{n-m}-\beta_3\beta_4-\beta_2\beta_1\\
   \alpha_1\beta_1-\beta_1\alpha_3,
   \alpha_3\beta_2-\beta_2\alpha_1, \alpha_3\beta_3-\beta_3\alpha_2\\
   \alpha_3^{i-n+m}\beta_2\beta_1 \quad (i:=\max\{m, n-m\})
\end{array}$$

 {\bf\parindent=0pt Case II: } $m=n-1$. In this case, $\alpha_1=\beta_1\beta_2$, $\alpha_2=\beta_4\beta_3$, $\alpha_3=\beta_3\beta_4+\beta_2\beta_1$, and the morphisms $\beta_i, i=1,2,3,4$ are irreducible. The algebra $\Lambda(n, n-1)$ is given by the following quiver with relations.
$$\begin{array}{c}
\xymatrix{{\bullet}\ar@<2pt>[r]^{\beta_1}&\bullet\ar@<2pt>[l]^{\beta_2}\ar@<2pt>[r]^{\beta_3}&\bullet\ar@<2pt>[l]^{\beta_4}\\
}\\
\beta_1\beta_3,\,  \beta_4\beta_2,  \, (\beta_2\beta_1)^{n-1}, \, (\beta_3\beta_4)^{n-1}\end{array}$$

{\bf\parindent=0pt Case III: } $m=1$.   In this case, $\alpha_3=0$ and the other morphisms above are irreducible. Then algebra $\Lambda(n, 1)$ is given by the following quiver with relations.
$$\begin{array}{c}
\xymatrix{{\bullet}\ar@(lu,ld)_(.4){\alpha_1}\ar@<2pt>[r]^{\beta_1}&\bullet\ar@<2pt>[l]^{\beta_2}\ar@<2pt>[r]^{\beta_3}&\bullet\ar@<2pt>[l]^{\beta_4}\ar@(ru,rd)^{\alpha_2}\\
}\\
\beta_1\beta_3, \, \beta_4\beta_2, \, {\alpha}_1 {\beta}_1, \, {\beta}_2 {\alpha}_1, \, {\alpha}_2{\beta}_4
, \, \beta_3\alpha_2\\
 \alpha_1^{n-1}-\beta_1\beta_2,\,  \alpha_2^{n-1}-\beta_4\beta_3, \, \beta_2\beta_1
\end{array}$$
One can calculate the Cartan matrix of $\Lambda(n, m)$:
$$\begin{bmatrix} m & m & m\\
m & n & 0\\
m & 0 &n \end{bmatrix}  (2m\leq n), \quad \begin{bmatrix} 3m-n & m & m\\
m & n & 0\\
m & 0 &n \end{bmatrix}  (2m > n).$$
 The construction above gives an ${\cal D}_1$-split triangle
$$\cpx{Y}\lra A\oplus \Sigma^{-1}A\lra \cpx{X}\lra \Sigma\cpx{Y}$$
in $\K{\modcat{A}}$.  An easy calculation shows that $\cpx{Y}$ is isomorphic in $\K{\modcat{A}}$ to the complex
$$0\lra A\lraf{\cdot x^{n-m}}A\lra 0.$$
Then the algebras $\Lambda(n, m)=\End_{\K{A}/\mathcal{G}}(\cpx{X}\oplus A\oplus \Sigma^{-1}A)$ and $\End_{\K{A}/\mathcal{G}}(\cpx{Y}\oplus A\oplus \Sigma^{-1}A)$ are derived equivalent.  Note that $\End_{\K{A}/\mathcal{G}}(\cpx{Y}\oplus A\oplus \Sigma^{-1}A)$ is just $\Lambda(n, n-m)$. That is, the algebra $\Lambda(n, m)$ is derived equivalent to $\Lambda(n, n-m)$ for all $1\leq m\leq n-1$.

\section*{Acknowledgements}
The research works of both authors are partially supported by NSFC. The first author was
partially supported by NSFC no. 11301398 and RFDP no. 20130141120035.
The second author
thanks  BNSF(1132005, KZ201410028033)  and the Fundamental Research Funds for the Central Universities for partial support. Both authors thank Prof. Steffen K\"{o}nig for warm hospitality and kind help during their visit to the University of Stuttgart in 2015.

\bibliographystyle{aomplain}
\bibliography{../refData}

\providecommand{\bysame}{\leavevmode\hbox to3em{\hrulefill}\thinspace}
\providecommand{\noopsort}[1]{}
\providecommand{\mr}[1]{\href{http://www.ams.org/mathscinet-getitem?mr=#1}{MR~#1}}
\providecommand{\zbl}[1]{\href{http://www.zentralblatt-math.org/zmath/en/search/?q=an:#1}{Zbl~#1}}
\providecommand{\jfm}[1]{\href{http://www.emis.de/cgi-bin/JFM-item?#1}{JFM~#1}}
\providecommand{\arxiv}[1]{\href{http://www.arxiv.org/abs/#1}{arXiv~#1}}
\providecommand{\doi}[1]{\url{http://dx.doi.org/#1}}
\providecommand{\MR}{\relax\ifhmode\unskip\space\fi MR }
\providecommand{\MRhref}[2]{%
  \href{http://www.ams.org/mathscinet-getitem?mr=#1}{#2}
}
\providecommand{\href}[2]{#2}
\begin{thebibliography}{10}

\bibitem{Aihara2012}
\bgroup\scshape{}T.~Aihara\egroup{} and \bgroup\scshape{}O.~Iyama\egroup{},
  Silting mutation in triangulated categories,  \emph{J. Lond. Math. Soc. (2)}
  \textbf{85} (2012), 633--668.

\bibitem{Auslander1980}
\bgroup\scshape{}M.~Auslander\egroup{} and \bgroup\scshape{}S.~O.
  Smal{\o}\egroup{}, Preprojective modules over Artin algebras,  \emph{J.
  Algebra} \textbf{66} (1980), 61--122.

\bibitem{Buan2006}
\bgroup\scshape{}A.~B. Buan\egroup{}, \bgroup\scshape{}R.~Marsh\egroup{},
  \bgroup\scshape{}M.~Reineke\egroup{}, \bgroup\scshape{}I.~Reiten\egroup{},
  and \bgroup\scshape{}G.~Todorov\egroup{}, Tilting theory and cluster
  combinatorics,  \emph{Adv. Math.} \textbf{204} (2006), 572--618.

\bibitem{Chen2013}
\bgroup\scshape{}Y.~Chen\egroup{}, Derived equivalences in {$n$}-angulated
  categories,  \emph{Algebr. Represent. Theory} \textbf{16} (2013), 1661--1684.

\bibitem{Geiss2013}
\bgroup\scshape{}C.~Geiss\egroup{}, \bgroup\scshape{}B.~Keller\egroup{}, and
  \bgroup\scshape{}S.~Oppermann\egroup{}, {$n$}-angulated categories,  \emph{J.
  Reine Angew. Math.} \textbf{675} (2013), 101--120.

\bibitem{Geis2006}
\bgroup\scshape{}C.~Gei{\ss}\egroup{}, \bgroup\scshape{}B.~Leclerc\egroup{},
  and \bgroup\scshape{}J.~Schr{{\"o}}er\egroup{}, Rigid modules over
  preprojective algebras,  \emph{Invent. Math.} \textbf{165} (2006), 589--632.

\bibitem{Happel1988}
\bgroup\scshape{}D.~Happel\egroup{}, \emph{Triangulated categories in the
  representation theory of finite dimensional algebras}, \textbf{119},
  Cambridge University Press, 1988.

\bibitem{Happel1998b}
\bgroup\scshape{}D.~Happel\egroup{} and \bgroup\scshape{}L.~Unger\egroup{},
  Complements and the generalized {N}akayama conjecture,  in \emph{Algebras and
  modules, {II} ({G}eiranger, 1996)}, \emph{CMS Conf. Proc.} \textbf{24}, Amer.
  Math. Soc., Providence, RI, 1998, pp.~293--310.

\bibitem{Hoshino2003}
\bgroup\scshape{}M.~Hoshino\egroup{} and \bgroup\scshape{}Y.~Kato\egroup{}, An
  elementary construction of tilting complexes,  \emph{J. Pure Appl. Algebra}
  \textbf{177} (2003), 159--175.

\bibitem{Hu2013a}
\bgroup\scshape{}W.~Hu\egroup{}, \bgroup\scshape{}S.~Koenig\egroup{}, and
  \bgroup\scshape{}C.~C. Xi\egroup{}, Derived equivalences from cohomological
  approximations and mutations of $\Phi$-Yoneda algebras,  \emph{Proc. Roy.
  Soc. Edinburgh Sect. A} \textbf{143} (2013), 589--629.

\bibitem{Hu2011}
\bgroup\scshape{}W.~Hu\egroup{} and \bgroup\scshape{}C.~C. Xi\egroup{}, ${\cal
  {D}}$-split sequences and derived equivalences,  \emph{Adv. Math.}
  \textbf{227} (2011), 292--318.

\bibitem{Hu2013}
\bgroup\scshape{}W.~Hu\egroup{} and \bgroup\scshape{}C.~C. Xi\egroup{}, Derived
  equivalences for $\Phi$-Auslander-Yoneda algebras,  \emph{Trans. Amer. Math.
  Soc.} \textbf{365} (2013), 5681--5711.

\bibitem{Iyama2014a}
\bgroup\scshape{}O.~Iyama\egroup{} and \bgroup\scshape{}M.~Wemyss\egroup{},
  Maximal modifications and {A}uslander-{R}eiten duality for non-isolated
  singularities,  \emph{Invent. Math.} \textbf{197} (2014), 521--586.

\bibitem{Keller1994}
\bgroup\scshape{}B.~Keller\egroup{}, Deriving DG categories,  \emph{Ann. Sci.
  {\'E}cole Norm. Sup.} \textbf{27} (1994), 63--102.

\bibitem{Ladkani2010}
\bgroup\scshape{}S.~Ladkani\egroup{}, Perverse equivalences, BB-tilting,
  mutations and applications,  \emph{Preprint} (2010).

\bibitem{Rickard1989a}
\bgroup\scshape{}J.~Rickard\egroup{}, Morita theory for derived categories,
  \emph{J. London Math. Soc.} \textbf{39} (1989), 436--456.

\bibitem{Rudakov1990}
\bgroup\scshape{}A.~N. Rudakov\egroup{}, Exceptional collections, mutations and
  helices,  in \emph{Helices and vector bundles}, \emph{London Math. Soc.
  Lecture Note Ser.} \textbf{148}, Cambridge Univ. Press, Cambridge, 1990,
  pp.~1--6.

\end{thebibliography}

\bigskip

Yiping Chen

\medskip
School of Mathematics and Statistics, Wuhan University, 430072 Wuhan,  China.

{\tt Email: ypchen@whu.edu.cn}

\bigskip
Wei Hu

\medskip
School of Mathematical Sciences,  Laboratory of Mathematics and Complex Systems, Beijing Normal University, 100875 Beijing,  China

\smallskip
and

\smallskip
 Beijing Center for Mathematics and Information Interdisciplinary Sciences, 100048 Beijing, China

 {\tt Email: huwei@bnu.edu.cn}

\end{document}